\documentclass[reqno]{amsart}
\usepackage[left=1in, right=1in, top=1in]{geometry}

\usepackage{etoolbox}

\makeatletter
\let\old@tocline\@tocline
\let\section@tocline\@tocline
\newcommand{\section@dotsep}{4.5}
\newcommand{\subsection@dotsep}{4.5}
\patchcmd{\@tocline}
  {\hfil}
  {\nobreak
     \leaders\hbox{$\m@th
        \mkern \section@dotsep mu\hbox{.}\mkern \section@dotsep mu$}\hfill
     \nobreak}{}{}
\let\section@tocline\@tocline
\let\@tocline\old@tocline

\patchcmd{\@tocline}
  {\hfil}
  {\nobreak
     \leaders\hbox{$\m@th
        \mkern \subsection@dotsep mu\hbox{.}\mkern \subsection@dotsep mu$}\hfill
     \nobreak}{}{}
\let\subsection@tocline\@tocline
\let\@tocline\old@tocline

\let\old@l@section\l@section
\let\old@l@subsection\l@subsection

\def\@tocwriteb#1#2#3{%
  \begingroup
    \@xp\def\csname #2@tocline\endcsname##1##2##3##4##5##6{%
      \ifnum##1>\c@tocdepth
      \else \sbox\z@{##5\let\indentlabel\@tochangmeasure##6}\fi}%
    \csname l@#2\endcsname{#1{\csname#2name\endcsname}{\@secnumber}{}}%
  \endgroup
  \addcontentsline{toc}{#2}%
    {\protect#1{\csname#2name\endcsname}{\@secnumber}{#3}}}%

\newlength{\@tocsectionindent}
\newlength{\@tocsubsectionindent}
\newlength{\@tocsubsubsectionindent}
\newlength{\@tocsectionnumwidth}
\newlength{\@tocsubsectionnumwidth}
\newlength{\@tocsubsubsectionnumwidth}
\newcommand{\settocsectionnumwidth}[1]{\setlength{\@tocsectionnumwidth}{#1}}
\newcommand{\settocsubsectionnumwidth}[1]{\setlength{\@tocsubsectionnumwidth}{#1}}
\newcommand{\settocsubsubsectionnumwidth}[1]{\setlength{\@tocsubsubsectionnumwidth}{#1}}
\newcommand{\settocsectionindent}[1]{\setlength{\@tocsectionindent}{#1}}
\newcommand{\settocsubsectionindent}[1]{\setlength{\@tocsubsectionindent}{#1}}
\newcommand{\settocsubsubsectionindent}[1]{\setlength{\@tocsubsubsectionindent}{#1}}

\renewcommand{\l@section}{\section@tocline{1}{\@tocsectionvskip}{\@tocsectionindent}{}{\@tocsectionformat}}%
\renewcommand{\l@subsection}{\subsection@tocline{1}{\@tocsubsectionvskip}{\@tocsubsectionindent}{}{\@tocsubsectionformat}}%
\renewcommand{\l@subsubsection}{\subsubsection@tocline{1}{\@tocsubsubsectionvskip}{\@tocsubsubsectionindent}{}{\@tocsubsubsectionformat}}%
\newcommand{\@tocsectionformat}{}
\newcommand{\@tocsubsectionformat}{}
\newcommand{\@tocsubsubsectionformat}{}
\expandafter\def\csname toc@1format\endcsname{\@tocsectionformat}
\expandafter\def\csname toc@2format\endcsname{\@tocsubsectionformat}
\expandafter\def\csname toc@3format\endcsname{\@tocsubsubsectionformat}
\newcommand{\settocsectionformat}[1]{\renewcommand{\@tocsectionformat}{#1}}
\newcommand{\settocsubsectionformat}[1]{\renewcommand{\@tocsubsectionformat}{#1}}
\newcommand{\settocsubsubsectionformat}[1]{\renewcommand{\@tocsubsubsectionformat}{#1}}
\newlength{\@tocsectionvskip}
\newcommand{\settocsectionvskip}[1]{\setlength{\@tocsectionvskip}{#1}}
\newlength{\@tocsubsectionvskip}
\newcommand{\settocsubsectionvskip}[1]{\setlength{\@tocsubsectionvskip}{#1}}
\newlength{\@tocsubsubsectionvskip}
\newcommand{\settocsubsubsectionvskip}[1]{\setlength{\@tocsubsubsectionvskip}{#1}}

\patchcmd{\longrightarrowcsection}{\indentlabel}{\makebox[\@tocsectionnumwidth][l]}{}{}
\patchcmd{\longrightarrowcsubsection}{\indentlabel}{\makebox[\@tocsubsectionnumwidth][l]}{}{}
\patchcmd{\longrightarrowcsubsubsection}{\indentlabel}{\makebox[\@tocsubsubsectionnumwidth][l]}{}{}

\newcommand{\@sectypepnumformat}{}
\renewcommand{\contentsline}[1]{%
  \expandafter\let\expandafter\@sectypepnumformat\csname @toc#1pnumformat\endcsname%
  \csname l@#1\endcsname}
\newcommand{\@tocsectionpnumformat}{}
\newcommand{\@tocsubsectionpnumformat}{}
\newcommand{\@tocsubsubsectionpnumformat}{}
\newcommand{\setsectionpnumformat}[1]{\renewcommand{\@tocsectionpnumformat}{#1}}
\newcommand{\setsubsectionpnumformat}[1]{\renewcommand{\@tocsubsectionpnumformat}{#1}}
\newcommand{\setsubsubsectionpnumformat}[1]{\renewcommand{\@tocsubsubsectionpnumformat}{#1}}
\renewcommand{\@tocpagenum}[1]{%
  \hfill {\mdseries\@sectypepnumformat #1}}

\let\oldappendix\appendix
\renewcommand{\appendix}{%
  \leavevmode\oldappendix%
  \addtocontents{toc}{%
    \protect\settowidth{\protect\@tocsectionnumwidth}{\protect\@tocsectionformat\sectionname\space}%
    \protect\addtolength{\protect\@tocsectionnumwidth}{2em}}%
}
\makeatother

\makeatletter
\settocsectionnumwidth{2em}
\settocsubsectionnumwidth{2.5em}
\settocsubsubsectionnumwidth{3em}
\settocsectionindent{1pc}%
\settocsubsectionindent{\dimexpr\@tocsectionindent+\@tocsectionnumwidth}%
\settocsubsubsectionindent{\dimexpr\@tocsubsectionindent+\@tocsubsectionnumwidth}%
\makeatother

\settocsectionvskip{0pt}
\settocsubsectionvskip{0pt}
\settocsubsubsectionvskip{0pt}

\usepackage{mathrsfs}
\usepackage{mathrsfs}
\usepackage{amsfonts}
\usepackage{comment}
\usepackage{cases}
\usepackage{latexsym}
\usepackage{amsmath}
\usepackage[all]{xy}
\usepackage[pdftex]{graphicx}
\usepackage[colorinlistoftodos]{todonotes}
\usepackage{stmaryrd}
\usepackage{amsfonts}
\usepackage{amsmath,amssymb,amscd,bbm,amsthm,mathrsfs,dsfont}
\usepackage[most]{tcolorbox}
\usepackage[notref, draft, notcite]{showkeys}

\usepackage{tikz}

\usepackage{pgflibraryarrows}
\usepackage{pgflibrarysnakes}

\usepackage[numbers,sort&compress]{natbib}
\usepackage[pagebackref]{hyperref}
\usepackage{hypernat}

\usepackage{color, hyperref}
\definecolor{pink}{rgb}{0.9,0.0,0.9}
\hypersetup{colorlinks, breaklinks,
            linkcolor=red,urlcolor=pink,
            anchorcolor=blue, citecolor=pink}

\usepackage{fancyhdr}
\usepackage{amsxtra,ifthen}
\usepackage{verbatim}

\usepackage{tikz,xcolor,hyperref}
\definecolor{lime}{HTML}{A6CE39}
\DeclareRobustCommand{\orcidicon}{
\begin{tikzpicture}
\draw[lime, fill=lime] (0,0)
circle[radius=0.16]
node[white]{{\fontfamily{qag}\selectfont \tiny \.{I}D}}; 
\end{tikzpicture}
\hspace{-2mm}
}
\foreach \x in {A, ..., Z}{%
\expandafter\xdef\csname orcid\x\endcsname{\noexpand\href{https://orcid.org/\csname orcidauthor\x\endcsname}{\noexpand\orcidicon}}
}

\usepackage{enumitem}

\usepackage{fancyhdr}
\usepackage{amsxtra,ifthen}
\usepackage{verbatim}

\numberwithin{equation}{section}

\usepackage{extarrows}
\newcommand{\eqdef}{\xlongequal{\text{def}}}

\theoremstyle{plain}
\newtheorem{theorem}{Theorem}[section]
\newtheorem{lemma}[theorem]{Lemma}
\newtheorem{proposition}[theorem]{Proposition}

\newtheorem{corollary}[theorem]{Corollary}

\theoremstyle{definition}

\newtheorem{remark}[theorem]{Remark}

\allowdisplaybreaks[4]

\makeatletter              
\let\c@equation\c@theorem  
\makeatother

\def \lim{\hbox{\lower0.9ex\hbox{$\buildrel{\lower1.1ex\hbox{lim}}
\over{\textstyle \leftarrow}$}}\,}

\title[Characterization of $n$-Lie Derivations on Generalized Matrix Algebras]
{Characterization of $n$-Lie Derivations on Generalized Matrix Algebras}

\author{Xinfeng Liang\hspace{-1.5mm}\orcidA{}}
\address{Liang: School of Mathematics and Big data, Anhui University of Science \& Technology, 232001, Huainan, P. R. China}
\email{\href{mailto:xfliang@aust.edu.cn}{xfliang@aust.edu.cn}}

\author{Minghao Wang}
\address{Wang: School of Mathematics and Big data, Anhui University of Science \& Technology, 232001, Huainan, P. R. China}
\email{\href{mailto:mhwang2023@163.com}{mhwang2023@163.com}}

\author{Feng Wei\hspace{-1.5mm}\orcidC{}}\thanks{*Corresponding Author}
\address{Wei$^\ast$ (Corresponding Author): School of Mathematics and Statistics, Beijing
Institute of Technology, Beijing, 100081, P. R. China}
\email{\href{mailto:daoshuo@hotmail.com}{daoshuo@hotmail.com}} \email{\href{maito:daoshuowei@gmail.com}{daoshuowei@gmail.com}}

\begin{document}

\begin{abstract}
 \noindent  The principal objective of this paper is to determine the structure of $n$-Lie derivations ($n\geq 3$) on generalized matrix algebras.
It is shown that under certain mild assumptions, every $n$-Lie derivation can be decomposed
into the sum of an extremal $n$-derivation and an $n$-linear central-valued mapping.
As direct applications, we provide complete characterizations of $n$-Lie derivations
on both full matrix algebras and triangular algebras.
\end{abstract}

\date{\today}



\maketitle


\section{Introduction}
\label{xxsec1}

Let $\mathcal{A}$ be an associative algebra defined over a commutative ring
$\mathcal{R}$, and denote the center of $\mathcal{A}$ by $\mathcal{Z}(\mathcal{A})$.
Throughout this paper, we assume that the ring $\mathcal{R}$ is $2$-torsionfree.
For the reader's convenience, we need to state and introduce some concrete linear mappings. A linear mapping
$\rho: \mathcal{A}\longrightarrow \mathcal{A}$ satisfying the relation
$$
\rho(xy)=\rho(x)y+x\rho(y) ~(\text{resp.}~ \rho([x,y])=[\rho(x),y]+[x,\rho(y))]
$$
is called a \textit{derivation} (resp. a \textit{Lie derivation}), where 
$[x,y]=xy-yx$ is the usual Lie product
for all $x,y\in \mathcal{A}$. Let $n$ be a positive integer. For an $n$-linear mapping
$\varrho:\underbrace{\mathcal{A}\times\cdots\times\mathcal{A}}_n\longrightarrow \mathcal{A}$,
if each component is a derivation (resp. Lie derivation), it is referred to as an
\textit{$n$-derivation} (resp. \textit{$n$-Lie derivation}).
Clearly, biderivations (resp. bi-Lie derivations) are special cases of
$n$-derivations (resp.  $n$-Lie derivations) whenever $n=2$.
It is clear that each derivation is a Lie derivation, and thus every
$n$-derivation is an $n$-Lie derivation, but the converse statements are not in general true. If an
$n$-derivation $\varrho:\underbrace{\mathcal{A}\times\cdots\times\mathcal{A}}_n\longrightarrow \mathcal{A}$
can be expressed in the form $\varrho(x_1,\cdots, x_n)=\varrho(x_{w(1)},\cdots, x_{w(n)})$, it is called a \textit{permuting
$n$-derivation} for all $x_1,\cdots, x_n\in \mathcal{A}$ and $w\in \mathfrak{S}_n  $, where
$\mathfrak{S}_n$ is the symmetric group of order $n$. In particular, whenever $n=2$, 
a permuting $2$-derivation is also known as a \textit{symmetric biderivation}.

The notion of biderivation was proposed by Maksa \cite{Maksa1980, Maksa1987},
it became an essential technique in the functional identity theory of noncommutative algebras since then.
There has been increasing interests in determining the structures
of biderivations on various associative algebras, see \cite{Benkovic2009, DuWang3, Eremita2017, Ghosseiri2013, Ghosseiri2017, 
LiangWei, LiuLiu2021, RenLiang2022, Wang2016}. Simultaneously, plenty of works are contributed to 
exploring biderivations on different types Lie and Leibniz algebras, see \cite{BresarZhao2018, ChangChen, ChangChenZhou1, ChangChenZhou2, ChenYaoZhao, 
DiLaRosa, DingTang, Eremita2021, HanWangXia, LiuGuoZhao, Mancini, Tang, TangLi, TangZhong, WangYu, WangYuChen}, 
which are exactly corresponding to bi-Lie derivations on associative algebras. 
For a detailed exposition on biderivations, we refer the reader to \cite{Bresar2004}.

For an arbitrary non-commutative algebra $\mathcal{A}$,
inner biderivations are of considerable importance,
as they serve as a measure tool for the non-commutativity of
$\mathcal{A}$, see \cite{LiangZhao2023,Benkovic2009,Jabeen2022,AlghazzawiArifSobhi2023}. 
A bilinear mapping
$\sigma:\mathcal{A}\times\mathcal{A}\longrightarrow \mathcal{A}$ is said to be an \textit{inner biderivation} if it is of the form 
$\sigma(x,y)=\lambda[x,y]$ for all $x,y\in \mathcal{A}$ and some $\lambda\in \mathcal{Z}(\mathcal{A})$. 
A natural common generalization of inner biderivations is the so-called extremal
$n$-derivations. An $n$-linear mapping $\upsilon:\underbrace{\mathcal{A}\times \cdots \times \mathcal{A}}_n\longrightarrow \mathcal{A}$ is said to be an 
\textit{extremal $n$-derivation} if it satisfies the relation 
$\upsilon(x_1,\cdots, x_n)=[x_1, \cdots, [x_n, x_0]\cdots]$ for all $x_1, \cdots, x_n\in \mathcal{A}$ and
some $x_0\notin \mathcal{Z}(\mathcal{A})$ with $[[\mathcal{A}, \mathcal{A}], x_0]=0$. 
The extremal $n$-derivations are indispensable tools when people embark on investigating the structures of
multilinear Lie-type derivations, see \cite{WangWangDu, Jabeen2022, Jabeen2024, LiangGuo2024}.

In recent years, some people paid their special attentions on multilinear derivations on triangular or generalized matrix algebras. 
Notably, it was Benkovi\v c who first proposed the concept of extremal biderivations and initiated the 
study of biderivations on triangular algebras \cite{Benkovic2009}. It was shown that under certain assumptions, 
every biderivation of a triangular algebra is the sum of an extremal biderivation and an inner biderivation. 
This result is immediately applied to (block) upper triangular matrix algebras and Hilbert space nest algebras.
Wang \cite{Wang2016} also addressed the decomposition question of biderivations on triangular rings 
and obtained the same decomposition with the former. However, Wang employed the techniques from the maximal left ring of quotients 
\cite{Utumi1956, Eremita2015}. Jabeen \cite{Jabeen2022} extended Benkovi\v c's and Wang's results the background  of
$n$-Lie derivations ($n\geq3$). She first gave a deep insight towards the structure of $3$-Lie derivations on a triangular algebra. 
And then, using mathematical induction method, she observed that under certain restrictions, 
every $n$-Lie derivation ($n\geq 3$) on a triangular algebra split into an extremal $n$-derivation and an $n$-linear central mapping. 
People gradually turn their views to multilinear derivations on generalized matrix algebras.  
Du and Wang \cite{DuWang3} demonstrated the structure of biderivations on generalized matrix algebras. 
It was shown that under certain conditions, very biderivation on a generalized matrix algebra is the sum of an 
extremal and an inner biderivation. They also considered the question when a biderivation on a generalized 
matrix algebra is an inner biderivation. As a consequence they proved that every biderivation of a 
full matrix algebra over a unital algebra is inner. In \cite{Jabeen2024, WangWangDu}, Jabeen and Wang etal independently  
investigated the question of when an $n$-derivations ($n\geq3$) on a triangular algebra (or generalized matrix algebra) 
will be an extremal $n$-derivation.
At the same time, multilinear Lie-type derivations have also attracted our attentions. Motivated by the previous works, 
the first and third authors \cite{LiangWei} characterized the structure of bi-Lie derivations (also known as $2$-Lie derivations)
on triangular algebras. Under some midl assumptions, every bi-Lie derivation can be decomposed into 
the sum of an inner biderivation, an extremal biderivation and a central bilinear mapping.

Taking into accounts the results regarding to multilinear derivations and Lie derivations on triangular algebras (or generalized matrix algebras)
\cite{Benkovic2009, DuWang3, LiangWei, Jabeen2022, Jabeen2024, Wang2016, WangWangDu},
people naturally ask: what can we say about the multilinear Lie derivations
on generalized matrix algebras ? This is our principal objective of the current work.   
The first step is to describe the structure of $3$-Lie derivations
on generalized matrix algebras from two different point of views. Under distinct assumptions, it is shown that every $3$-Lie derivation is the sum of
an extremal $3$-derivation and a $3$-linear central-valued mapping. 
Then by mathematical induction approach, we eventually achieve the structure characterization of
$n$-Lie derivations, showing that every $n$-Lie derivation is the sum of an extremal
$n$-derivation and an $n$-linear central-valued mapping. It is worthy to note that we determine the structure 
of $n$-Lie derivations on generalized matrix algebras under different perspectives. 
Roughly speaking, our main results are as follows.








\begin{tcolorbox}[breakable, enhanced, blanker, left=3mm,right=3mm,
borderline west={2pt}{0pt}{cyan}]

 \newtheorem*{theoremA}{Theorem A}
\begin{theoremA} {\rm (Proposition \ref{xxsec3.1} + Theorem \ref{xxsec4.1})}
Let $\mathcal{G}=\left[
\begin{array}
[c]{cc}%
A & M\\
N & B\\
\end{array}
\right]$ be a generalized matrix algebrabe over a commutative ring $\mathcal{R}$ and
$\varphi:\underbrace{\mathcal{G}\times \cdots \times \mathcal{G}}_n\longrightarrow\mathcal{G}$ be a
$n$-Lie derivation on $\mathcal{G}\, \, (n\geq 3)$. Suppose that $\mathcal{G}$ satisfies the following conditions:
\begin{enumerate}
\item [{\rm (1)}] $\pi_{A}(\mathcal{Z}(\mathcal{G}))=\mathcal{Z}(A)$ and $\pi_{B}(\mathcal{Z}(\mathcal{G}))=\mathcal{Z}(B);$
\item [{\rm (2)}] either $A$ or $B$ does not contain nonzero central ideals;
\item [{\rm (3)}] if $\alpha a = 0, \alpha\in \mathcal{Z}(\mathcal{G}), 0 \neq a \in \mathcal{G}$, then $\alpha = 0;$
\item [{\rm (4)}] If $MN = 0 = NM$,
            then at least one of the algebras $A$ and $B$ is noncommutative;
\item [{\rm (5)}] every special pair of bimodule homomorphisms has standard form.
\end{enumerate}
Then $\varphi$ is of the form $\varphi=\kappa+\psi$,
where $\kappa$ is an extremal $n$-derivation such that
$\kappa(x_1,x_2,\cdots,x_n) =[x_1,[x_2,[\cdots,[x_n, X_0]\cdots]]]$ for all $x_1, x_2, \cdots, x_n\in \mathcal{G}$, 
$\psi$ is an $n$-linear central-valued mapping vanishing on commutators in each component, 
$X_0=e\varphi(e,e,\cdots,e)f+(-1)^{n}f\varphi(e,e,\cdots,e)e$. 
\end{theoremA}

\end{tcolorbox}

\begin{tcolorbox}[breakable, blanker,left=3mm,right=3mm,
borderline west={2pt}{0pt}{purple}]

 \newtheorem*{theoremB}{Theorem B}
\begin{theoremB}[Proposition \ref{xxsec3.17} + Theorem \ref{xxsec4.3}]
Let $\mathcal{G}=\left[
\begin{array}
[c]{cc}%
A & M\\
N & B\\
\end{array}
\right]$ be a generalized matrix algebra over a commutative ring $\mathcal{R}$
and $\varphi:\underbrace{\mathcal{G}\times \cdots \times \mathcal{G}}_n\longrightarrow\mathcal{G}$ be an
$n$-Lie derivation on $\mathcal{G}\, \, (n\geq 3)$. Suppose that
\begin{enumerate}
\item[{\rm (1)}] $\pi_{A}(\mathcal{Z}(\mathcal{G}))=\mathcal{Z}(A)$ and $\pi_{B}(\mathcal{Z}(\mathcal{G}))=\mathcal{Z}(B);$
\item[{\rm (2)}] either $A$ or $B$ does not contain nonzero central ideals;
\item[{\rm (3)}] For each $n\in N$, the condition $Mn = 0 = n M$ implies $n=0$;
\item[{\rm (4)}] For each $m\in M$, the condition $Nm = 0 = m N$ implies $m=0$;
\item[{\rm (5)}] each special pair of bimodule homomorphisms has standard form.
\end{enumerate}
Then $\varphi$ has the form $\varphi=\kappa+\psi$, where $\kappa$ is an extremal $n$-derivation such that
$\kappa(x_1,x_2,\cdots,x_n) =[x_1,[x_2,[\cdots,[x_n, X_0]\cdots]]]$ for all $x_1, x_2, \cdots, x_n\in \mathcal{G}$, 
$\psi$ is an $n$-linear central-valued mapping vanishing on commutators in each component, 
$X_0=e\varphi(e,e,\cdots,e)f+(-1)^{n}f\varphi(e,e,\cdots,e)e$. 
\end{theoremB}
\end{tcolorbox}







 This paper is organized as follows. The first section \ref{xxsec1} is devoted to a brief introduction of research background
and necessary concepts related to the current topics. Some basic facts and examples of generalized matrix algebras are provided in section \ref{xxsec2}. 
The third section \ref{xxsec3} should be the mainbody of this work and is occupied by the characterization 
of $3$-Lie derivations on generalized matrix algebras (see Proposition \ref{xxsec3.1} and Proposition \ref{xxsec3.17}). 
We would like to point out that we characterize the structure of $3$-Lie derivations from two different perspectives, 
laying the groundwork for the subsequent main theorems. The two main results with respect to $n$-Lie derivations, Theorem \ref{xxsec4.1} and 
Theorem \ref{xxsec4.3}, are stated and presented with solid details in section \ref{xxsec4}. 
Both of them shows that under mild assumptions, every $n$-Lie derivation on a generalized matrix algebra is 
the sum of an extremal $n$-derivation and a central-valued mapping. 
We end up this work with direct applications of the above-mentioned main results.

\vspace{2mm}
\section{Generalized Matrix Algebras and Examples}
\label{xxsec2}

Let us start with the definition of generalized matrix algebras. We also provide some basic facts and examples with respect to  
generalized matrix algebras in this section. 

Let $\mathcal{R}$ be a commutative ring with identity. A \textit{Morita context} is a sextuple $(A, B, M, N, \Phi_{MN},\Psi_{NM})$,
where $A$ and $B$ are unital $\mathcal{R}$-algebras, $_AM_B$ and $_BN_A$ are bimodules, and $\Phi_{MN}: M\underset {B}{\otimes} N\longrightarrow A$ and
$\Psi_{NM}: N\underset {A}{\otimes} M\longrightarrow B$
are bimodule homomorphisms satisfying the following commutative diagrams:
$$
\xymatrix{ M \underset {B}{\otimes} N \underset{A}{\otimes} M
\ar[rr]^{\hspace{8pt}\Phi_{MN} \otimes I_M} \ar[dd]^{I_M \otimes
\Psi_{NM}} && A
\underset{A}{\otimes} M \ar[dd]^{\cong} \\  &&\\
M \underset{B}{\otimes} B \ar[rr]^{\hspace{10pt}\cong} && M }
\hspace{6pt}{\rm and}\hspace{6pt} \xymatrix{ N \underset
{A}{\otimes} M \underset{B}{\otimes} N
\ar[rr]^{\hspace{8pt}\Psi_{NM}\otimes I_N} \ar[dd]^{I_N\otimes
\Phi_{MN}} && B
\underset{B}{\otimes} N \ar[dd]^{\cong}\\  &&\\
N \underset{A}{\otimes} A \ar[rr]^{\hspace{10pt}\cong} && N
\hspace{2pt}.}
$$
We refer the reader to \cite{Morita} about Morita contexts.
Given a Morita context $(A, B, M, N,$ $ \Phi_{MN},
\Psi_{NM})$, the set of formal matrices
$$
\mathcal{G}=\left[
\begin{array}
[c]{cc}%
A & M\\
N & B\\
\end{array}
\right]=\left\{ \left[
\begin{array}
[c]{cc}%
a& m\\
n & b\\
\end{array}
\right] \vline a\in A, m\in M, n\in N, b\in B \right\}
$$
constitutes an $\mathcal{R}$-algebra under matrix-like addition and multiplication,
provided at least one of the bimodules $M$ or $N$ is nonzero. This resulted algebra is 
called a \textit{generalized matrix algebra of order $2$}.
Such kind of construction naturally extends to generalized matrix algebras of order $n>2$.
However, as shown in \cite[Example 2.2]{LW1}, every generalized matrix algebra of order
$n\geq 2$ is, up to isomorphism, equivalent to one of order $2$.
If either $M=0$ or $N=0$, $\mathcal{G}$ will reduce to be an upper triangular algebra
$$\mathcal{T^U}=\mathcal{T}(A, M, B)=
\left[
\begin{array}
[c]{cc}%
A & M\\
O & B\\
\end{array}
\right].
$$
or a lower triangular algebra
$$
\mathcal{T_L}=\mathcal{T}(A, N, B)=
\left[
\begin{array}
[c]{cc}%
A & O\\
N & B\\
\end{array}
\right].
$$
The framework of generalized matrix algebras provides a unifying perspective,
encompassing both triangular algebras and full matrix algebras $\mathcal{M}_n(\mathcal{R})$.
This approach allows us to carry out a systematic study for linear and nonlinear mappings
on these algebras under a single theoretical framework, as demonstrated in \cite{LW1, LW2, LWW, XW1, XW2}.

We briefly list some classical examples of generalized matrix algebras that will
be invoked in later sections (Sections \ref{xxsec3}--\ref{xxsec4}).
As these examples are well-documented in the literature \cite{XW1,XW2,LW1,LW2},
we only mention their names without elaboration.
\begin{enumerate}
\item[{\rm (1)}] Unital algebras with nontrivial idempotents;
\item[{\rm (2)}] Full matrix algebras $\mathcal{M}_{n\times n}(\mathcal{R})$ over a commutative ring $\mathcal{R}$;
\item[{\rm (3)}] Inflated algebras;
\item[{\rm (4)}] Triangular algebras, such as upper or lower triangular matrix algebras, block upper (or lower) triangular matix algebras and nest algebras over a Hilbert space;
\item[{\rm (5)}] Factor von Neumann algebra acting on a Hilbert space;
\item[{\rm (6)}] von Neumann algebra with no central summand of type $I_1$;
\item[{\rm (7)}] The algebra $\mathcal{B}(X)$ of all bounded linear operators over a Banach space $X$;
\item[{\rm (8)}] Standard operator algebras over a Banach space.
\end{enumerate}

These generalized matrix algebras are ubiquitous in associative algebras and 
noncommutative Noetherian algebras, owing to their versatility
and intuitive structural properties. Despite their prevalence,
the investigation of linear mappings on these algebras has not received enough concerns.
The systematic study of such linear mappings was pioneered by Krylov and his teammates,
who approached the subject from a classification perspective \cite{Krylov1}.
This landmark work has inspired numerous contributions,
leading to some significant progresses in related areas
(see \cite{Benkovic, BenkovicSirovnik, DuWang2, DuWang3, LW1, LW2, LWW, WangWang, XW1, XW2}).
Beyond linear mappings of generalized matrix algebras, 
Krylov and his collaborators also investigated their representation theory,
homological behaviors, and $K$-theory aspects, see
\cite{Krylov1, Krylov2, Krylov3, Krylov4, KrylovNor1, KrylovNor2, KrylovTuganbaev1, KrylovTuganbaev2, KrylovTuganbaev3, KrylovTuganbaev4, KrylovTuganbaev5}.

Throughout this paper, we consider the generalized matrix algebra
of order $2$ 
$$
\mathcal{G}=\left[
\begin{array}
[c]{cc}%
A & M\\
N & B
\end{array}
\right], 
$$
which is associated with the Morita context $(A, B, _AM_B, _BN_A, \Phi_{MN}, \Psi_{NM})$. At least, one of the bimodules $M$ or $N$ is nonzero. 
However, we would like to point out that the following computational strategy 
$$
\begin{aligned}
\mathcal{G} & =\left[
\begin{array}
[c]{cc}%
A & M\\
N & B
\end{array}
\right]\\
& =\left[
\begin{array}
[c]{cc}%
A & O\\
O & O
\end{array}
\right]+\left[
\begin{array}
[c]{cc}%
O & M\\
O & O
\end{array}
\right]+ \left[
\begin{array}
[c]{cc}%
O & O\\
N & O
\end{array}
\right]+ \left[
\begin{array}
[c]{cc}%
O & O\\
O & B
\end{array}
\right] \\
& \stackrel{\cong}{=} A+M+N+B
\end{aligned}
$$
is heavily dependent and frequently adopted without specific explanations. 
Let
$$
e=\left[
\begin{array}
[c]{cc}%
1_{A} & 0\\
0 & 0\\
\end{array}
\right]  \, \, \, \, \text{and}  \, \, \, \, 
f=\left[
\begin{array}
[c]{cc}
0 & 0\\
0 & 1_{B}\\
\end{array}
\right]
$$
 be idempotents in the algebra $\mathcal{G}$ with $1_{A}$ and $1_{B}$
being the identity elements of algebras $A$ and $B$, respectively.
These elements satisfy relation $e+f=I$, where $I$ denotes the identity of
$\mathcal{G}$.
We always assume that $M$ is faithful both as a
left $A$-module and as a right $B$-module,
while no additional conditions are imposed on $N$.
The center of $\mathcal{G}$ is defined to be
$$
\mathcal{Z(G)}=\left\{ \left[
\begin{array}
[c]{cc}%
a & 0\\
0 & b
\end{array}
\right] \vline \hspace{3pt} am=mb, \hspace{3pt} na=bn,\ \forall\
m\in M, \hspace{3pt} \forall n\in N \right\}.
$$
Two $\mathcal{R}$-linear canonical projection homomorphisms
$\pi_A:\mathcal{G}\longrightarrow A$ and $\pi_B:\mathcal{G}\longrightarrow
B$ can be established via the mappings:
$$
\pi_A: \left[
\begin{array}
[c]{cc}%
a & m\\
n & b\\
\end{array}
\right] \longmapsto a \quad \text{and} \quad \pi_B: \left[
\begin{array}
[c]{cc}%
a & m\\
n & b\\
\end{array}
\right] \longmapsto b.
$$
By the preceding characterization of $\mathcal{Z}(\mathcal{G})$, it is easy to observe that:
$\pi_A\left(\mathcal{Z(G)}\right)$ is a subalgebra of $\mathcal{Z}(A)$ and
that $\pi_B\left(\mathcal{Z(G)}\right)$ is a subalgebra of $\mathcal{Z}(B)$.
Furthermore, there exists an algebra isomorphism $\eta:\pi_A(\mathcal{Z(G)})\longrightarrow
\pi_B(\mathcal{Z(G)})$ is an algebraic isomorphism such that
$am=m\eta(a)$ and $na=\eta(a)n$ for all $a\in \pi_A(\mathcal{Z(G)}), m\in M, n\in N$.
Alternatively, it is equivalent to saying that $mb=\eta^{-1}(b)m$ and $bn=n\eta^{-1}(b)$ for all 
$b\in \pi_B(\mathcal{Z(G)}), m\in M, n\in N$.

An $(A,B)$-bimodule endomorphism $\mathcal{F}:M\longrightarrow M$ is said to be
\textit{standard form} if $\mathcal{F}(m) = \omega_0 m + m \omega_1$
holds true for all $m\in M$, some central elements $\omega_0 \in \mathcal{Z}(A)$ and $\omega_1 \in \mathcal{Z}(B)$.
Correspondingly, a $(B,A)$-bimodule endomorphism $\mathcal{E}:N\longrightarrow N$ admits
the \textit{standard form} if it is of the form $\mathcal{E}(n) = n \varpi_0 + \varpi_1 n$
for all $n\in N$, certain central elements $\varpi_0 \in \mathcal{Z}(A)$ and $\varpi_1 \in \mathcal{Z}(B)$.

A special pair $(\mathcal{F}, \mathcal{E})$ of bimodule homomorphisms $\mathcal{F}:M\longrightarrow M$ and $\mathcal{E}:N\longrightarrow N$
 is called \emph{special} if they satisfy the compatibility condition:
$$
\mathcal{F}(m)n+m\mathcal{E}(n)=0=n\mathcal{F}(m)+\mathcal{E}(n)m
$$
for all $m\in M$ and $n\in N$. Such an special pair $(\mathcal{F}, \mathcal{E})$
possesses the \emph{standard form} if there exist central elements $\omega_0 \in \mathcal{Z}(A)$
and $\omega_1 \in \mathcal{Z}(B)$ such that:
$$
\mathcal{F}(m)=\omega_0m+m\omega_1~ \text{and} ~\mathcal{E}(n)=-n\omega_0-\omega_1n
$$
for all $m\in M$ and $n\in N$.

\vspace{2mm}
\section{Key Techniques: $3$-Lie Derivations}
\label{xxsec3}

Although Du and Wang \cite{DuWang2} previously determined the structure of
Lie derivations on generalized matrix algebras, the structure of bi-Lie derivation on these algebras remain untouched and mysterious,
because their characterization need to overcome substantial difficulties,
and consequently will not be handled under the current work.

Building upon the methodology developed in \cite{Jabeen2022} for analyzing
$n$-Lie derivations on triangular algebras,
we employ mathematical induction on the parameter $n$,
with the induction established for $n=3$. Accordingly,
we will finish the structural description of $3$-Lie derivations on generalized matrix algebras in this section.

\begin{tcolorbox}[blanker,left=3mm,right=3mm,
borderline west={2pt}{0pt}{orange}]

\begin{proposition}\label{xxsec3.1}
Let $\mathcal{G}=\left[
\begin{array}
[c]{cc}%
A & M\\
N & B\\
\end{array}
\right]$ be a generalized matrix algebra over a commutative ring $\mathcal{R}$
and $\varphi:\mathcal{G}\times \mathcal{G}\times \mathcal{G}\longrightarrow\mathcal{G}$ be a $3$-Lie derivation on $\mathcal{G}$. Suppose that
\begin{enumerate}
\item [{\rm (1)}] $\pi_{A}(\mathcal{Z}(\mathcal{G}))=\mathcal{Z}(A)$ and $\pi_{B}(\mathcal{Z}(\mathcal{G}))=\mathcal{Z}(B);$
\item [{\rm (2)}] either $A$ or $B$ does not contain nonzero central ideals;
\item [{\rm (3)}] if $\alpha a = 0, \alpha\in \mathcal{Z}(\mathcal{G}), 0 \neq a \in \mathcal{G}$, then $\alpha = 0;$
\item [{\rm (4)}] If $MN = 0 = NM$, then at least one of the algebras $A$ and $B$ is noncommutative;
\item [{\rm (5)}] every special pair of bimodule homomorphisms has standard form.
\end{enumerate}
Then $\varphi$ is of the form $\varphi=\kappa+\psi$, where $\kappa$ is an extremal $3$-derivation such that
$\kappa(x, y, z) =[x, [y, [z, X_0]]]$ for all $x, y, z\in \mathcal{G}$, $\psi$ is a $3$-linear central-valued mapping 
vanishing on commutators in each component, $X_0= e\varphi(e, e, e)f + (-1)^3f\varphi(e, e, e)e$. 
\end{proposition}
\end{tcolorbox}

To finish the proof of this proposition, we need to make some preparations in advance.

\begin{tcolorbox}[blanker,left=3mm,right=3mm,
borderline west={2pt}{0pt}{green}]

\begin{lemma}\label{xxsec3.2}{\rm \cite[Lemma 3.1]{LiangWei}}
Let $\mathcal{A}$ be an associative algebra over a commutative ring $\mathcal{R}$
 and $\varphi: \mathcal{A}\times \mathcal{A}\longrightarrow \mathcal{A}$ be a Lie biderivation on $\mathcal{A}$,
then $\varphi$ satisfies the following relation
$$
[\varphi(x,y),[v,u]]+[\varphi(x,v),[u,y]]=[\varphi(u,y),[x,v]]+[\varphi(u,v),[x,y]]     \eqno{({\rm LB})}
$$
for all $x, y, u, v\in \mathcal{A}$.
\end{lemma}
\end{tcolorbox}

It should be remarked that special pair of bimodule homomorphisms play a crucial role for the proof of Proposition \ref{xxsec3.1}.
We now bring two fundamental lemmas that will be essential for our subsequent analysis and discussion.

\begin{tcolorbox}[blanker,left=3mm,right=3mm,
borderline west={2pt}{0pt}{green}]

\begin{lemma}\label{xxsec3.3}{\rm \cite[Proposition 3.3]{DuWang3}}
Let $\mathcal{G}=\left[
\begin{array}
[c]{cc}%
A & M\\
N & B\\
\end{array}
\right]$ be a generalized matrix algebra over a commutative ring $\mathcal{R}$ satisfying the following conditions:
\begin{enumerate}
\item[{\rm (1)}] for each $n\in N$ the condition $Mn=0=nM$ implies $n=0$;
\item[{\rm (2)}] every $(A, B)$-bimodule homomorphism of $M$ has the standard form.
\end{enumerate}
Then each special pair $(\mathcal{F}, \mathcal{E})$
of bimodule homomorphisms $\mathcal{F}:M\longrightarrow M$
and $\mathcal{E}:N\longrightarrow N$ is of the standard form.
\end{lemma}
\end{tcolorbox}

\begin{tcolorbox}[blanker,left=3mm,right=3mm,
borderline west={2pt}{0pt}{green}]

\begin{lemma}\label{xxsec3.4}{\rm \cite[Proposition 3.4.]{DuWang3}}
Let $\mathcal{G}=\left[
\begin{array}
[c]{cc}%
A & M\\
N & B\\
\end{array}
\right]$ be a generalized matrix algebra over a commutative ring $\mathcal{R}$. If every derivation on 
$\mathcal{G}$ is inner, then every special pair $(\mathcal{F}, \mathcal{E})$ of bimodule homomorphisms $\mathcal{F}: M\longrightarrow  M$
and $ \mathcal{E}: N\longrightarrow N$ has the standard form.
\end{lemma}
\end{tcolorbox}

By comparing extremal $n$-derivations ($n>2$) with
extremal $2$-derivations we find whether a multilinear mapping
$\kappa_n (x_1, x_2,\cdots,x_n)=[x_1, [x_2,\cdots, [x_n, X_0]\cdots]]$
constitutes an extremal $n$-derivation heavily depends on the existence of an element $X_0\in \mathcal{G}$ 
satisfying condition $[[\mathcal{G}, \mathcal{G}],X_0]=0$.
The following proposition provides the existence condition and structure form of extremal $n$-derivations ($n\geq 2$). 
The proof is omitted due to its similarity to that of \cite[Proposition 4.1.]{DuWang3}.

\begin{tcolorbox}[blanker,left=3mm,right=3mm,
borderline west={2pt}{0pt}{orange}]

\begin{proposition}\label{xxsec3.5}
Let $\mathcal{G}=\left[
\begin{array}
[c]{cc}%
A & M\\
N & B\\
\end{array}
\right]$ be a generalized matrix algebra.  If there
exists $m_0 \in M$ and $n_0 \in N$ with $m_0\neq 0$ or $n_0 \neq 0$ such that
\begin{enumerate}
\item [{\rm (i)}] $[A, A]m_0=0=m_0[B, B]$ and $n_0[A, A]=0=[B, B]n_0$,
\item [{\rm (ii)}] $m_0N = 0 = Nm_0$ and $n_0M = 0 = Mn_0$.
\end{enumerate}
Then the equation $[[\mathcal{G}, \mathcal{G}],  X_0]=0$ holds true, where $X_0=m_0+(-1)^n n_0$.
\end{proposition}
\end{tcolorbox}

\begin{tcolorbox}[breakable, blanker,left=3mm,right=3mm,
borderline west={2pt}{0pt}{gray}]

\begin{remark}\label{xxsec3.6}
Let $\mathcal{G}=\left[
\begin{array}
[c]{cc}%
A & M\\
N & B\\
\end{array}
\right]$ be a generalized matrix algebra. Suppose that there exists an element $X_0\in \mathcal{G}$ such that
$[[\mathcal{G}, \mathcal{G}], X_0]=0$. Let us define a mapping
$$
\begin{aligned}
\nu: \underbrace{\mathcal{G}\times\cdots\times\mathcal{G}}_n & \longrightarrow \mathcal{G}\\
(x_1, x_2, \cdots, x_n) & \longmapsto [x_1, [x_2,\cdots, [x_n, X_0]\cdots ]], \, \, \, \, \, \forall x_1, x_2, \cdots, x_n\in \mathcal{G} . 
\end{aligned}
$$
Then $\nu$ is a permuting $n$-derivation of $\mathcal{G}$.
\end{remark}
\end{tcolorbox}

The proof of Remark \ref{xxsec3.6} is essentially identical with that of \cite[Remark 1]{WangWangDu}, 
and is therefore omitted here. The object defined therein $\nu$ continues to be referred to as an extremal $n$-derivation.
By invoking of \cite[Remark 1]{WangWangDu} again, we can make a reasonable 
deformation for the related mapping. Let us transform the $3$-Lie derivation in Proposition \ref{xxsec3.1} into
another much more simpler $3$-Lie derivation, see Proposition \ref{xxsec3.7} in below. For this purpose,
we will prove that the $3$-linear mapping
$\kappa:{\mathcal G}\times {\mathcal G}\times {\mathcal G}\longrightarrow {\mathcal G}$
defined in Proposition \ref{xxsec3.7}
is an extremal 3-derivation of a generalized matrix algebra $\mathcal{G}=\left[
\begin{array}
[c]{cc}%
A & M\\
N & B\\
\end{array}
\right].
$

\begin{tcolorbox}[breakable, blanker,left=3mm,right=3mm,
borderline west={2pt}{0pt}{orange}]

\begin{proposition}\label{xxsec3.7}
Let $\mathcal{G}=\left[
\begin{array}
[c]{cc}%
A & M\\
N & B\\
\end{array}
\right]$ be a generalized matrix algebra and $\varphi :\mathcal{G} \times \mathcal{G}\times \mathcal{G} \longrightarrow \mathcal{G}$
be a $3$-Lie derivation on $\mathcal{G}$. Suppose that 
\begin{enumerate}
\item [{\rm (1)}] $\pi_{A}(\mathcal{Z}(\mathcal{G}))=\mathcal{Z}(A)$ and $\pi_{B}(\mathcal{Z}(\mathcal{G}))=\mathcal{Z}(B);$
\item [{\rm (2)}] either $A$ or $B$ does not contain nonzero central ideals.
\end{enumerate}
If $\varphi (e, e, e) \ne 0$,
then $\varphi=\kappa+\psi$, where $\kappa$ is an extremal $3$-derivation such that $\kappa(x,y,z)=[x, [y, [z, X_0]]]$
for all $x, y, z\in \mathcal{G}$, $\psi$ is a $3$-Lie derivation such that 
$\psi(e, e, e)\in \mathcal{Z}(\mathcal{G})$, $X_0= e\varphi(e, e, e)f+ (-1)^3f\varphi(e, e, e)e$. 
\end{proposition}
\end{tcolorbox}

\begin{proof}

Let us complete the proof of this proposition via a series of claims.

\textbf{Claim 1.} With notations as above, we have $\varphi(x,y,0)=\varphi(x,0,z)=\varphi(0,y,z)=0$
for all $x ,y, z\in \mathcal{G}$.

Here, we exclusively provide the proof of $\varphi(x,y,0)=0$ for all $x, y\in \mathcal{G}$. 
For any $x ,y\in \mathcal{G}$, we have
$\varphi(x,y,0)=\varphi(x,y,[0,0])=[\varphi(x,y,0),0]+[0,\varphi(x,y,0)]=0$.
The other two relations can be obtained in an analogous way.

\textbf{Claim 2.} With notations as above, we have
\begin{enumerate}
  \item [(i)] $e\varphi(e,a_{1},a_{2})f=a_{2}\varphi(e,a_{1},e)f=a_{1}\varphi(e,e,a_{2})f=a_{1}a_{2}\varphi(e,e,e)f=a_{2}a_{1}\varphi(e,e,e)f$,
  \item [(ii)] $f\varphi(e,a_{1},a_{2})e=f\varphi(e,a_{1},e)a_{2}=f\varphi(e,e,a_{2})a_{1}=f\varphi(e,e,e)a_{1}a_{2}=f\varphi(e,e,e)a_{2}a_{1}$, 
  \item [(iii)] $e\varphi(e,b_{1},b_{2})f=-e\varphi(e,b_{1},e)b_{2}=-e\varphi(e,e,b_{2})b_{1}=e\varphi(e,e,e)b_{1}b_{2}=e\varphi(e,e,e)b_{2}b_{1}$,
  \item [(iv)]$f\varphi(e,b_{1},b_{2})e=-b_{2}\varphi(e,b_{1},e)e=-b_{1}\varphi(e,e,b_{2})e=b_{2}b_{1}\varphi(e,e,e)e=b_{1}b_{2}\varphi(e,e,e)e$
\end{enumerate}
for all $a_{1}, a_{2}\in A$ and $b_{1}, b_{2} \in B$.

We only prove the conclusions $(\text{i})-(\text{ii})$. The remaining conclusions
$(\text{iii})-(\text{iv})$ follow by using analogous arguments.

With the help of above {\bf Claim 1}, for all $a\in A$, we get
$$
\begin{aligned}
0&=\varphi(e,e,0)=\varphi(e,e,[e,a])=[\varphi(e,e,e),a]+[e,\varphi(e,e,a)]\\
&=\varphi(e,e,e)a-a\varphi(e,e,e)+e\varphi(e,e,a)-\varphi(e,e,a)e.
\end{aligned}\eqno{(3.0)}
$$
Multiplying both sides of (3.0) by $e$ and $f$ yields $e\varphi(e,e,a)f=a\varphi(e,e,e)f$.
Premultiplying by $f$ and postmultiplying by $e$ in (3.0) gives $f\varphi(e,e,a)e=f\varphi(e,e,e)a$.

Integrating the above computational techniques, all results can be summarized as follows:
$$
\begin{aligned}
e\varphi(e,e,a)f&=e\varphi(e,a,e)f=e\varphi(a,e,e)f=a\varphi(e,e,e)f, \\
f\varphi(e,e,a)e&=f\varphi(e,a,e)e=f\varphi(a,e,e)e=f\varphi(e,e,e)a,\\
e\varphi(e,e,b)f&=e\varphi(e,b,e)f=e\varphi(b,e,e)f=-e\varphi(e,e,e)b,\\
f\varphi(e,e,b)e&=f\varphi(e,b,e)e=f\varphi(b,e,e)e=-b\varphi(e,e,e)e.
\end{aligned}\eqno{(3.1)}
$$
for all $a\in A, b\in B$.

Applying the definition of $3$-Lie derivation yields 
$$
\begin{aligned}
0 & =\varphi(e,a_{1},0)\\
& =\varphi(e,a_{1},[a_{2},e])\\
&= [\varphi(e,a_{1},a_{2}),e]+[a_{2},\varphi(e,a_{1},e)]\\
&=\varphi(e,a_{1},a_{2})e-e\varphi(e,a_{1},a_{2})+a_{2}\varphi(e,a_{1},e)-\varphi(e,a_{1},e)a_{2}
\end{aligned}\eqno{(3.2)}
$$
for all $a_{1}, a_{2}\in A$. Multiplying both sides of (3.2) by $e$ and $f$ together with (3.1) gives
$$
e\varphi(e,a_{1},a_{2})f=a_{2}\varphi(e,a_{1},e)f=a_{2}a_{1}\varphi(e,e,e)f
$$
for all $a_{1}, a_{2}\in A$. Premultiplying (3.2) by $f$ and postmultiplying by $e$,  using (3.1), leads to
$$
f\varphi(e,a_{1},a_{2})e=f\varphi(e,a_{1},e)a_{2}=f\varphi(e,e,e)a_{1}a_{2}
$$
for all $a_{1}, a_{2}\in A$. For the second component, employing the aforementioned method give rise to
$$
e\varphi(e,a_{1},a_{2})f=a_{1}\varphi(e,e,a_{2})f=a_{1}a_{2}\varphi(e,e,e)f
$$
and
$$
f\varphi(e,a_{1},a_{2})e=f\varphi(e,e,a_{2})a_{1}=f\varphi(e,e,e)a_{2}a_{1}
$$
for all $a_{1}, a_{2} \in A$.
The preceding analysis establishes the validity of $(\text{i})$ and $(\text{ii})$.

\textbf{Claim 3.} With notations as above, we obtain
\begin{enumerate}
  \item [(i)] $e\varphi(e,e,m)e=e\varphi(e,m,e)e=e\varphi(m,e,e)e=-m\varphi(e,e,e)e$, 
  \item [(ii)] $f\varphi(e,e,m)f=f\varphi(e,m,e)f=f\varphi(m,e,e)f=f\varphi(e,e,e)m$, 
  \item [(iii)] $e\varphi(e,e,n)e=e\varphi(e,n,e)e=e\varphi(n,e,e)e=-e\varphi(e,e,e)n$, 
  \item [(iv)] $f\varphi(e,e,n)f=f\varphi(e,n,e)f=f\varphi(n,e,e)f=n\varphi(e,e,e)f$, 
  \item [(v)] the mappings $e\varphi(e,e,m)f, e\varphi(e,m,e)f$ and $ e\varphi(m,e,e)f$ are all $(A,B)$-bimodule homomorphisms,
  while the mappings $f\varphi(e,e,n)e, f\varphi(e,n,e)e$ and $ f\varphi(n,e,e)e$ are all $(B,A)$-bimodule homomorphisms
\end{enumerate}
for all $m\in M, n\in N$.

The proof of $(\text{i})$ is merely provided here, and the remaining cases can be achieved in an analogous manner. 
For all $m\in M$, we have
$$
\begin{aligned}
\varphi(e,e,m)& =\varphi(e,e,[e,m])\\
&= [\varphi(e,e,e),m]+[e,\varphi(e,e,m)]\\
&= \varphi(e,e,e)m-m\varphi(e,e,e)+e\varphi(e,e,m)-\varphi(e,e,m)e.
\end{aligned}\eqno{(3.3)}
$$
Multiplying both sides of (3.3) by $e$ yields $e\varphi(e,e,m)e=-m\varphi(e,e,e)e$.
Multiplying both sides of  (3.3) by $f$ produces $f\varphi(e,e,m)f=f\varphi(e,e,e)m$ for all $m\in M$.

Employing the same technique, one can reach the relations $e\varphi(m,e,e)e=e\varphi(e,m,e)e=-m\varphi(e,e,e)e$ and 
$f\varphi(e,m,e)f=f\varphi(m,e,e)f=f\varphi(e,e,e)m$ for all $m\in M$.

For computational convenience, let us set
$$
m_0=e\varphi(e,e,e)f, n_0=f\varphi(e,e,e)e.
$$

\textbf{Claim 4.} With notations as above, we get
\begin{enumerate}
  \item [(i)] $[A,A]m_{0}=m_{0}[B,B]=0~~\text{and} ~~[B,B]n_{0}=n_{0}[A,A]=0$;
  \item [(ii)] $m_{0}N=Nm_{0}=0~~\text{and}~~n_{0}M=Mn_{0}=0.$
\end{enumerate}

For any $a_{1},a_{2}\in A$, $b_{1},b_{2}\in B$,
it follows from the relations $(\text{i})$ and $(\text{iii})$ in {\bf Claim 2} that
$$
[a_{1},a_{2}]m_0=[a_{1},a_{2}]e\varphi(e,e,e)f=a_{1}a_{2}\varphi(e,e,e)f-a_{2}a_{1}\varphi(e,e,e)f=0  \eqno{(3.4)}
$$
and
$$
m_{0}[b_{1},b_{2}]=e\varphi(e,e,e)f[b_{1},b_{2}]=e\varphi(e,e,e)b_{1}b_{2}-e\varphi(e,e,e)b_{2}b_{1}=0.  \eqno{(3.5)}
$$

From an alternative perspective,
for all $a_{1},a_{2}\in A$, $b_{1},b_{2}\in B$,
it follows from $(\text{ii})$ and $(\text{iv})$ in {\bf Claim 2} that
$$
[b_{1},b_{2}]n_0=[b_1, b_2](f\varphi(e,e,e)e)=(b_{1}b_{2}\varphi(e,e,e)e-b_{2}b_{1}\varphi(e,e,e)e)=0  \eqno{(3.6)}
$$
and
$$
n_0[a_1, a_2]=f\varphi(e,e,e)e[a_1, a_2]=(f\varphi(e,e,e)a_1a_2-f\varphi(e,e,e)a_2a_1)=0.  \eqno{(3.7)}
$$
Thus we arrive at the result $(\text{i})$ with the help of equations $(3.4)-(3.7)$.

Let us next consider the second result $(\text{ii})$. By the definition of $3$-Lie derivation, we get
$$
\begin{aligned}
\varphi(e,m_{1},m_{2}) &=\varphi(e,[e, m_{1}],m_{2})\\
&= [\varphi(e, e, m_{2}),m_{1}]+[e, \varphi(e,m_{1},m_{2})]\\
&= \varphi(e,e,m_{2})m_{1}-m_{1}\varphi(e,e,m_{2})+e\varphi(e,m_{1},m_{2})-\varphi(e,m_{1},m_{2})e.
\end{aligned}
$$
for all $m_1, m_2\in M$. Multiplying the above equation both sides by $e$ and $f$ leads to
$$
e\varphi(e,e,m_{2})em_{1}=m_{1}f\varphi(e,e,m_{2})f \eqno{(3.8)}
$$
for all $m_1, m_2\in M$. 

By invoking the definition of $3$-Lie derivation again, we see that 
$$
\begin{aligned}
\varphi(e,n_{1},m_{2})& =\varphi(e, [n_{1},e], m_{2})\\
&=[\varphi(e,n_{1},m_{2}),e]+[n_{1},\varphi(e,e,m_{2})]\\
&=\varphi(e,n_{1},m_{2})e-e\varphi(e,n_{1},m_{2})+n_{1}\varphi(e,e,m_{2})-\varphi(e,e,m_{2})n_{1}.
\end{aligned}
$$
for all $n_1\in N, m_2\in M$. Multiplying both sides of the above equation by $f$ and $e$ produces 
$$
n_{1}e\varphi(e,e,m_{2})e=f\varphi(e,e,m_{2})fn_{1} \eqno{(3.9)}
$$
for all $n_1\in N, m_2\in M$.   
  
Combining (3.8) with (3.9) yields
$$
e\varphi(e,e,m_{2})e+f\varphi(e,e,m_{2})f\in \mathcal{Z(G)}. 
$$
for all $m_2\in M$. This implies that 
$$
e\varphi(e,e,m_{2})e\in \mathcal{Z}(A), \, \, \, \, \, f\varphi(e,e,m_{2})f\in \mathcal{Z}(B).\eqno{(3.10)}
$$

Let us now prove that $Mn_0=\{mn_0  \mid ~\text{for all}~ m\in M\}=\{0\}=n_0M=\{n_0m  \mid ~\text{for all}~m\in M\}$. 
In view of (3.10) and the result (i) in {\bf Claim 3} we assert that
$$
amn_0=-am\varphi(e,e,e)e=e\varphi(e,e,am)e\in \mathcal{Z}(A)      \eqno{(3.11)}
$$
for all $a\in A, m\in M$. 

In light of $(3.10)$, we know that 
$$
\begin{aligned}
\mathcal{Z}(A) & \ni e\varphi(e,e,am)e =e\varphi(e,e,[a,m])e\\
&=e([\varphi(e,e,a),m]+[a,\varphi(e,e,m)])e\\
&=e([\varphi(e,e,a),m])e=-m\varphi(e,e,e)a=-mn_0a.\\
\end{aligned} \eqno{(3.12)}
$$
for all $a\in A, b\in B$ and $m\in M$. With the help of (3.11) and (3.12), we conclude that the set
$Mn_0=\{mn_0 \mid ~\text{for all}~m\in M\}$ is a central ideal of $A$. Without loss of generality, 
we might suppose that algebra $A$ does not contain nonzero central ideals. This shows that $Mn_0=\{0\}$.
And hence, (3.11) becomes $0=amn_0=-e\varphi(e, e, am)e$ for all $a\in A, m\in M$.     
Note that the relation (3.8) results in $m_{1}f\varphi(e,e,m_{2})f=e\varphi(e,e,m_{2})em_{1}=m_{1}\eta(e\varphi(e,e,m_{2})e)$ for all $m_1, m_2\in M$. 
In light of the faithfulness of $M$ as $(A, B)$-bimodule, we see that $f\varphi(e,e,m)f=\eta(e\varphi(e,e,m)e)$ for all $m\in M$.
Combining (3.10) with the  relation (ii) in {\bf Claim 3} gives 
we have
$$
n_0m=f\varphi(e,e,e)em=f\varphi(e,e,m)f=\eta(e\varphi(e,e,m)e)=0 \eqno{(3.13)}
$$
for all $m\in M$. (3.13) implies that $n_0M=\{0\}$.

Let us now demonstrate the proof of the relation $m_{0}N=0=Nm_{0}$.
For arbitrary elements $n_{1}, n_{2}\in N$, we see that
$$
\begin{aligned}
\varphi(e,n_{1},n_{2}) & =\varphi(e,[n_{1},e],n_{2})\\
&=[\varphi(e,n_{1},n_{2}),e]+[n_{1},\varphi(e,e,n_{2})]\\
&=\varphi(e,n_{1},n_{2})e-e\varphi(e,n_{1},n_{2})+n_{1}\varphi(e,e,n_{2})-\varphi(e,e,n_{2})n_{1}.\\
\end{aligned}
$$
Multiplying both sides of the above equation by $f$ and $e$ produces 
$$
n_{1}e\varphi(e,e,n_{2})e=f\varphi(e,e,n_{2})fn_{1} \eqno{(3.14)}
$$
for all $n_{1}, n_{2}\in N$. Furthermore, for any $m_{1}\in M, n_{2}\in N$, we have
$$
\begin{aligned}
\varphi(e,m_{1},n_{2}) & =\varphi(e,[m_{1},e],n_{2})\\
&=[\varphi(e,m_{1},n_{2}),e]+[m_{1},\varphi(e,e,n_{2})]\\
&=\varphi(e,m_{1},n_{2})e-e\varphi(e,m_{1},n_{2})+m_{1}\varphi(e,e,n_{2})-\varphi(e,e,n_{2})m_{1}.\\
\end{aligned}
$$
Multiplying both sides of the above equation by $e$ and $f$
yields 
$$
m_{1}f\varphi(e,e,n_{2})f=e\varphi(e,e,n_{2})em_{1}\eqno{(3.15)}
$$ 
for all $m_{1}\in M, n_{2}\in N$. Taking into account the relation (3.14) together with (3.15) we get
$$
e\varphi(e,e,n_{2})e+f\varphi(e,e,n_{2})f\in \mathcal{Z(G)}
$$
for all $n_{1}, n_{2}\in N$. Thus 
$$
e\varphi(e,e,n_{2})e\in \mathcal{Z}(A), \, \, \, \, \, f\varphi(e,e,n_{2})f\in \mathcal{Z}(B). \eqno{(3.16)}
$$
According to (3.16) and the relation (iii) in  {\bf Claim 3}, we have
$$
m_0na=e\varphi(e,e,e)na=-e\varphi(e,e,na)e\in \mathcal{Z}(A)\eqno{(3.17)}
$$
for all $a\in A, n\in N$, 

On the other hand,  by (3.16) and the relation (i) in {\bf Claim 2} it follows that 
$$
\begin{aligned}
\mathcal{Z}(A) & \ni e\varphi(e,e,na)e=e\varphi(e,e,[n,a])e\\
&=e([\varphi(e,e,n),a]+[n,\varphi(e,e,a)])e\\
&=e \varphi(e,e,n)ae-ea \varphi(e,e,n)e+e[n, \varphi(e,e,a)]e\\
&= e \varphi(e,e,n)ea -a e\varphi(e,e,n)e +e[n , \varphi(e,e,a)]e\\
&=[ e\varphi(e,e,n) e , a]+e[n, \varphi(e,e,a)]e\\
&=e[n,\varphi(e,e,a)]e=-a\varphi(e,e,e)n=-am_0n\\
\end{aligned}
$$
for all $a\in A, n\in N$. Thus $m_0N=\{m_0n \mid ~\text{for all}~ n\in N\}$ is a central ideal of $A$. 
Without loss of generality, we might 
assume that algebra $A$ does not contain nonzero central ideals. This shows that $m_0N=\{0\}$.
Thus (3.17) becomes $0=m_0na=e\varphi(e,e,e)na=-e\varphi(e,e,na)e$ for all $a\in A, n\in N$. 
Note that the relation (3.15) brings $e\varphi(e,e,n_2)em_{1}=m_{1}f\varphi(e,e,n_2)f=m_1\eta(e\varphi(e,e,n_2)e)$.
Considering the fact that the $M$ as $(A, B)$-bimodule is faithful, we immediately 
get $f\varphi(e,e,n_2)f=\eta(e\varphi(e,e,n_2)e)$ for all $n_2\in N$.
In view of (3.16) and the relation (iv) in {\bf Claim 3}, we arrive at 
$$
nm_0=n\varphi (e,e,e)f=f\varphi (e,e,n)f=\eta(e\varphi(e,e,n)e)=0
$$
for all $n\in N$. This shows that $Nm_0=\{0\}$. We eventually finish the proof of {\bf Claim 4}.

Let us set $X_0=e\varphi(e,e,e)f+(-1)^{3}f\varphi(e,e,e)e$.
Using {\bf Claim 4} and Proposition \ref{xxsec3.5}, a straightforward verification
shows that the element $X_0$ indeed satisfy the relation $[[\mathcal{G}, \mathcal{G}], X_0]=0$. 
Thus the induced $3$-linear mapping $\kappa (x, y, z)=[x, [y, [z, X_0]]]$
is an extremal $3$-Lie derivation on $\mathcal{G}$ for all $x, y, z\in \mathcal{G}$.

Let us establish a $3$-linear mapping $\psi=\varphi - \kappa$. Then 
$$
\begin{aligned}
\psi(e,e,e) &=\varphi(e,e,e)-\kappa (e,e,e)\\
&=\varphi(e,e,e)-[e, [e, [e, X_0]]]\\
&=\varphi(e, e,e)-e\varphi(e, e, e)f-f\varphi(e, e, e)e\\
&=e\varphi(e,e,e)e+f\varphi(e,e,e)f\in \mathcal{Z(G)}.
\end{aligned}
$$
Thus, $\psi(e,e,e)\in  \mathcal{Z(G)}$. It is not difficult to check that the
mapping $\psi=\varphi-\kappa$ is also a $3$-Lie derivation on $\mathcal{G}$ and satisfies
the relation $\psi(e,e,e)\in  \mathcal{Z(G)}$.
\end{proof}

As we see in Proposition \ref{xxsec3.7}, the well-establishes mapping $\psi=\varphi-\kappa$ is a $3$-Lie derivation
on the generalized matrix algebra $\mathcal{G}$ such that $\psi(e,e,e)\in  \mathcal{Z(G)}$. In fact, $\psi$ possesses much more stronger 
property $\psi(x,y,z)\in  \mathcal{Z(G)}$ for all $x, y, z\in \mathcal{G}$.
Let us show the validity of this conclusion from multiple perspectives.

\begin{tcolorbox}[breakable, enhanced, blanker, left=3mm,right=3mm,
borderline west={2pt}{0pt}{green}]

\begin{lemma}\label{xxsec3.8}
With notations as above, we have
\begin{enumerate}
  \item [{\rm (1)}] $\psi(I, y, z)\in \mathcal{Z(G)}$ and $e\psi(I, y, z)f=f\psi(I, y, z)e=0$;
  \item [{\rm (2)}] $\psi(y, I, z)\in \mathcal{Z(G)}$ and $e\psi(y, I, z)f=f\psi(y, I, z)e=0$;
  \item [{\rm (3)}] $\psi(y, z, I)\in \mathcal{Z(G)}$ and $e\psi(y, z, I)f=f\psi(y, z, I)e=0$
\end{enumerate}
 for all $y,z\in \mathcal{G}$.
\end{lemma}

\end{tcolorbox}

\begin{proof}
For any $x, y, z\in \mathcal{G}$, we have
$$
\begin{aligned}
0&=\psi(0,y,z)=\psi([I,x],y,z)\\
&=[\psi(I,y,z), x]+[I, \psi(x,y,z)]=[\psi(I,y,z), x].\\
\end{aligned}
$$
With notations as above, we have $\psi(I,y,z)\in \mathcal{Z(G)}$.
Thus $e\psi(I,y,z)f=f\psi(I,y,z)e=0$. The remaining results can be achieved in analogous computational approaches. 
\end{proof}

\begin{tcolorbox}[breakable, enhanced, blanker, left=3mm,right=3mm,
borderline west={2pt}{0pt}{gray}]

\begin{remark}\label{xxsec3.9}
With notations as above, for all $x, y, z\in A\cup B$,
the trilinear mapping $\psi: \mathcal{G}\times \mathcal{G}\times \mathcal{G}\longrightarrow \mathcal{G}$
defined by $\psi(x,y,z)=\varphi(x,y,z)-\kappa (x,y,z)$
is a 3-Lie derivation, which satisfies Condition $\psi(e,e,e)\in  \mathcal{Z(G)}$
and consequently fulfills the following equations $e\psi(e,e,e)f=0$ and $f\psi(e,e,e)e=0$.
Furthermore, by Lemma \ref{xxsec3.8} it follows that 
$$
\begin{aligned}
&e\psi(e,e,e)f=e\psi(e,e,f)f=e\psi(e,f,f)f=-e\psi(f,f,f)f\\
=&e\psi(f,e,f)f=e\psi(e,f,e)f=e\psi(f,e,e)f=e\psi(f,f,e)f=0.
\end{aligned}
$$
and that
$$
\begin{aligned}
&f\psi(e,e,e)e=f\psi(e,e,f)e=f\psi(e,f,f)e=-f\psi(f,f,f)e\\
=&f\psi(f,e,f)e=f\psi(e,f,e)e=e\psi(f,e,e)e=f\psi(f,f,e)e=0.
\end{aligned}
$$
\end{remark}
\end{tcolorbox}

\begin{tcolorbox}[breakable, enhanced, blanker, left=3mm,right=3mm,
borderline west={2pt}{0pt}{green}]

\begin{lemma}\label{xxsec3.10}
With notations as above, we observe 
$$
\psi(x, y, z)\in  \left[
\begin{array}
[c]{cc}%
A & O\\
O & B\\
\end{array}
\right]
$$
for all $x, y, z\in A \cup B$.
Furthermore, the following results 
$$
e\psi(x_1, y, z)e, e\psi(y,x_1, z)e, e\psi(y, z, x_1)e\in \mathcal{Z}(A)
$$ and
$$
f\psi(y, x_2, z)f, f\psi(x_2, y, z)f, f\psi(y, z, x_2)f\in \mathcal{Z}(B)
$$
hold true for all $x_1\in B, x_2\in A, y,z\in A\cup B$.
In particular, we have
\begin{enumerate}
  \item [{\rm (1)}] $\psi(y, e, z)=\psi(e,y, z)=\psi(y, z, e)\in \mathcal{Z(G)}$;
  \item [{\rm (2)}] $\psi(y, e, e)=\psi(e, y, e)=\psi(e, e, y)\in \mathcal{Z(G)}$;
  \item [{\rm (3)}] $\psi(e, e, e)=\psi(e, y, e)=\psi(e, e, y)\in \mathcal{Z(G)}$
\end{enumerate}
for all $y, z\in A\cup B$.
\end{lemma}
\end{tcolorbox}

\begin{proof}
Taking into account Remark \ref{xxsec3.9} together with the Lie derivation
structure of the first component, we get  
$$
\begin{aligned}
0=&\psi([x,e],e,e)=[\psi(x,e,e),e]+[x,\psi(e,e,e)]\\
=&\psi(x,e,e)e-e\psi(x,e,e)\\
\end{aligned}\eqno{(3.18)}
$$
for all $x\in A\cup B$.
Premultiply (3.18) by $e$ and postmultiply it the  by $f$ to obtain 
$e\psi(x,e,e)f=0$. Similarly, premultiply (3.18) by $f$ and postmultiply it by $e$
to yield $f\psi(x,e,e)e=0$. This implies that
$$
\psi(x, e, e) \in \left[
\begin{array}
[c]{cc}%
A & O\\
O & B\\
\end{array}
\right], \, \, \, \, \forall x\in A\cup B.
$$
For any $x\in A\cup B, m\in M$, we have
$$
\begin{aligned}
\psi(x,e,m)=&\psi(x,e,[e,m])\\
=&[\psi(x,e,e),m]+[e,\psi(x,e,m)]\\
=&\psi(x,e,e)m-m\psi(x,e,e)+e\psi(x,e,m)-\psi(x,e,m)e.\\
\end{aligned}\eqno{(3.19)}
$$
Premultiply (3.19) by $e$ and postmultiply it by $f$ to get
$$
e\psi(x,e,e)m=m\psi(x,e,e)f, \, \, \, \, \forall x\in A\cup B, m\in M.         
\eqno{(3.20)}
$$
Adopting an analogous computational technique and approach gives
$$
\begin{aligned}
\psi(x,e,n)=&\psi(x,e,[n,e])\\
=&[\psi(x,e,n),e]+[n,\psi(x,e,e)]\\
=&\psi(x,e,n)e-e\psi(x,e,n)+n\psi(x,e,e)-\psi(x,e,e)n
\end{aligned}\eqno{(3.21)}
$$
for all $n\in N$. Multiplying the left side of (3.21) by $f$ and multiplying the right side of (3.21) by $e$ 
produces 
$$
n\psi(x,e,e)e=f\psi(x,e,e)n, \, \, \, \, \forall x\in A\cup B, n\in N
\eqno{(3.22)}
$$
Combining (3.20) with (3.22) yields
$$
\psi(x, e, e) =
\left[
\begin{array}
[c]{cc}%
e\psi(x, e, e)e & 0\\
0 & f\psi(x, e, e)f\\
\end{array}
\right]\in \mathcal{Z}(\mathcal{G}), \, \, \, \, \forall x\in A\cup B.                 
\eqno{(3.23)}
$$

With the help of $(3.23)$, we have
$$
\begin{aligned}
0=&\psi(x,[y_1,e],e)=[\psi(x, y_1, e),e]+[y_1, \psi(x,e,e)]\\
=&\psi(x,y_1,e)e-e\psi(x,y_1,e)
\end{aligned}\eqno{(3.24)}
$$
for all $x, y_1\in A\cup B$. Premultiply (3.24) by $e$ and postmultiply it by $f$ to get
$$
e\psi(x, y_{1}, e)f=0, \, \, \, \, \forall x, y_1\in A\cup B.
\eqno{(3.25)}
$$ 
Similarly, multiplying the left side of (3.24) by $f$ and multiplying the right side of (3.24) by $e$ gives 
$$
f\psi(x, y_{1}, e)e=0, \, \, \, \, \forall x, y_1\in A\cup B.
\eqno{(3.26)}
$$
In view of (3.25) and (3.26), we know that
$$
\psi(x, y_{1}, e) \in  \left[
\begin{array}
[c]{cc}%
A & O\\
O & B\\
\end{array}
\right], \, \, \, \, \forall x, y_1\in A\cup B.
$$
For any $x,y\in A\cup B, m\in M$, we have
$$
\begin{aligned}
\psi(x,y,m)=&\psi(x, y,[e,m])\\
=&[\psi(x,y,e),m]+[e,\psi(x,y,m)]\\
=&\psi(x,y,e)m-m\psi(x,y,e)+e\psi(x,y,m)-\psi(x,y,m)e.
\end{aligned}\eqno{(3.27)}
$$
Multiplying the left side of (3.27) by $e$ and multiplying the right side of (3.27) by $f$ gives rise to
$$
e\psi(x,y,e)m=m\psi(x,y,e)f, \, \, \, \, \forall x, y\in A\cup B, m\in M.
\eqno{(3.28)}
$$
Using an analogous computational method, one can arrive at
$$
\begin{aligned}
\psi(x,y,n)=&\psi(x,y,[n,e])\\
=&[\psi(x,y,n),e]+[n,\psi(x,y,e)]\\
=&\psi(x,y,n)e-e\psi(x,y,n)+n\psi(x,y,e)-\psi(x,y,e)e\\
\end{aligned}\eqno{(3.29)}
$$
for all $x,y\in A\cup B$, $n\in N$.
Premultiply (3.29) by $f$ and postmultiply (3.29) by $e$ to get
$$
n\psi(x,y,e)e=f\psi(x,y,e)n, \, \, \, \, \forall x, y \in A\cup B, n\in N.
\eqno{(3.30)}
$$
Taking into account (3.28) and (3.30) we see that
$$
\psi(x, y, e) =
\left[
\begin{array}
[c]{cc}%
 e\psi(x, y, e)e & 0\\
0 & f\psi(x, y, e)f \\
\end{array}
\right] \in    \mathcal{Z}(\mathcal{G}), \, \, \, \, \forall x, y\in A\cup B.            
\eqno{(3.31)}
$$

By invoking the relation (3.31) it follows that 
$$
\begin{aligned}
0=&\psi(x, y ,[z,e])=[\psi(x, y , z),e]+[z, \psi(x, y, e)]\\
=&\psi(x, y , z)e-e\psi(x, y , z)
\end{aligned}\eqno{(3.32)}
$$
for all $x, y, z\in A\cup B$. Premultiplying (3.32) by $e$ and postmultiplying it by $f$ to 
obtain 
$$
e\psi(x, y, z)f=0, \, \, \, \, \forall x, y, z\in A\cup B.
\eqno{(3.33)}
$$ 
Similarly, multiplying the left side of (3.32) by $f$ and multiplying the right side of (3.32) by $e$ yields 
$$
f\psi(x, y, z)e=0, \, \, \, \, \forall x, y, z\in A\cup B.
\eqno{(3.34)}
$$
The relations (3.33) and (3.34) give
$$
\psi(x, y, z) \in  \left[
\begin{array}
[c]{cc}%
A & O\\
O & B\\
\end{array}
\right], \, \, \, \, \forall x, y, z\in A\cup B.
$$

For any $x\in A, y\in B$ and $z, w\in A\cup B$, we obtain $[x, y]=0$, from which
$$
\begin{aligned}
0=&\psi([x, y], z, w)=[\psi(x, z, w), y]+[x, \psi(y, z, w)]\\
=&\psi(x, z, w)y-y\psi(x, z, w)+x\psi(y, z, w)-\psi(y, z, w)x\\
\end{aligned}\eqno{(3.35)}
$$
follows. Furthermore, multiplying the left side of (3.35) by $f$ while multiplying the right side of it by $f$
yields $f\psi(x, z, w)y=y\psi(x, z, w)f$. Similarly, premultiplying and postmultiplying (3.35) by $e$
gives $x\psi(y, z, w)e=e\psi(y, z, w)x$.
Consequently, both $f\psi(x, z, w)f\in \mathcal{Z}(B)$ and $e\psi(y, z, w)e\in \mathcal{Z}(A)$ follow
for all $x\in A ,y\in B$ and $z, w\in A\cup B$.

Employing analogous computational methods one can arrive at the following relations:
$f\psi(z, x, w)f\in \mathcal{Z}(B)$ and $e\psi(z, y, w)e\in \mathcal{Z}(A)$;
$f\psi(z, w,x)f\in \mathcal{Z}(B)$ and $e\psi(z, w, y)e\in \mathcal{Z}(A)$
for all $x\in A ,y\in B$ and $z, w\in A\cup B$.

\end{proof}

\begin{tcolorbox}[breakable, enhanced, blanker, left=3mm,right=3mm,
borderline west={2pt}{0pt}{green}]

\begin{lemma}\label{xxsec3.11}
With notations as above, for any $x\in A\cup B$, $m\in M$, and $n\in N$, we have
\begin{enumerate}
  \item [{\rm (i)}]  $\psi(x,y_{1},m)=\psi(y_{1},x,m)=\psi(x,m,y_{1})=\psi(y_{1},m,x)=\psi(m,x,y_{1})=\psi(m,y_{1},x)=0$
  for all $y_{1}\in A\cup M \cup B$;
  \item [{\rm (ii)}] $\psi(x,y_{2},n)=\psi(y_{2},x,n)=\psi(x,n,y_{2})=\psi(y_{2},n,x)=\psi(n,x,y_{2})=\psi(n,y_{2},x)=0$
  for all $y_{2}\in A\cup N \cup B$.
\end{enumerate}
\end{lemma}
\end{tcolorbox}

\begin{proof}
For this lemma, we shall adopt a pairwise proof strategy, with Conclusions $\psi(x,y_{1},m)=0$
and $\psi(x,y_{2},n)=0$ being the primary objects of demonstration
for all $x\in A\cup B$, $m\in M$, $n\in N$, $y_{1}\in A\cup M \cup B$ and $y_{2}\in A\cup N \cup B$.
The remaining results follow in a similar way, and hence their proofs are omitted.


\textbf{Claim 1.} With notations as above, for arbitrary elements $m\in M, n\in N$, we have
$$
\psi(x,y,m)=e\psi(x,y,m)f\in M ~~\text{and}~~\psi(x,z,n)=f\psi(x,z,n)e\in N
$$
for all $x\in A\cup B, y\in A\cup B\cup M, z\in A\cup B\cup N$.

For any $x\in A\cup B, m\in M$, in light of Lemma \ref{xxsec3.10}, we know that
$$
\begin{aligned}
-\psi(x, m, e)&=\psi(x,[m, e], e)=[\psi(x,m,e),e]+[m,\psi(x,e,e)]\\
=&\psi(x,m,e)e-e\psi(x,m,e).
\end{aligned}\eqno{(3.36)}
$$
For the equation (3.36), performing simultaneous left/right multiplication by either $e$ or $f$
yields 
$$
e\psi(x,m,e)e=0\, \, \, \, \,  \text{and} \, \, \, \, \, f\psi(x,m,e)f=0, \, \, \, \, \, \forall x\in A\cup B, m\in M. 
$$ 
Premultiply (3.36) by $f$ and postmultiply it by $e$ to produce 
$$
f\psi(x,m,e)e=0, \, \, \, \, \, \forall x\in A\cup B, m\in M. 
$$
By the well-known Pierce decomposition theorem, we ultimately establish the validity of
$$
\psi(x,m,e)=e\psi(x,m,e)f\in M, \, \, \, \, \forall x\in A\cup B, m\in M.
$$
Using analogous computational methods, we see that 
$$
\psi(x,n,e)\in N,\, \, \, \, \forall x\in A\cup B, n\in N.
$$
The above calculations lead to
$$
\psi(x,m,e)\in M\, \, \, \, \, \text{and}\, \, \, \, \, \psi(x,n,e)\in N, \, \, \, \, \, \forall x\in A\cup B, m\in M, n\in N.
\eqno{(3.37)}
$$

Thus, for arbitrary elements $x\in A\cup B, y\in A\cup B\cup M$ and $m\in M$, we get
$$
\begin{aligned}
\psi(x,y,m)=&\psi(x,y,[e,m])=[\psi(x,y,e),m]+[e,\psi(x,y,m)]\\
=&\psi(x,y,e)m-m\psi(x,y,e)+e\psi(x,y,m)-\psi(x,y,m)e\\
=&\left\{
\begin{array}{rcl}
e\psi(x,y,e)m-m\psi(x,y,e)f+e\psi(x,y,m)-\psi(x,y,m)e& &{\text{if}~ y\in A\cup B},\\
e\psi(x,y,m)-\psi(x,y,m)e& &{\text{if}~ y\in M}.
\end{array} \right.
   \end{aligned}\eqno{(3.38)}
$$
Premultiplying and postmultiplying (3.38) by $e$ to produce  
$$
e\psi(x,y,m)e=0, \, \, \, \, \, \forall x\in A\cup B, y\in A\cup B\cup M, m\in M.
$$
Performing left multiplication by $f$ and right multiplication by $e$ for (3.38), 
we obtain 
$$
f\psi(x,y,m)e=0, \, \, \, \, \, \forall x\in A\cup B, y\in A\cup B \cup M, m\in M. 
$$
Multiplying the left and right sides of (3.38) by $f$, 
we get 
$$
f\psi(x,y,m)f=0, \, \, \, \, \, \forall x\in A\cup B, y\in A\cup B \cup M, m\in M.
$$
Applying the Pierce decomposition theorem again gives
$$
\psi(x,y,m)=e\psi(x,y,m)f\in M,  \, \, \, \, \, \forall x\in A\cup B, y\in A\cup B\cup M, m\in M.
$$

Following the aforementioned approach and computational strategies, the following relation 
$$
\psi(x,z,n)=f\psi(x,z,n)e\in N, \, \, \, \, \, \forall x\in A\cup B, z\in A\cup B\cup N, n\in N.
$$
can be eventually achieved.

\textbf{Claim 2.} With notations as above, we have
$$
\psi(x,z,m)=0\, \, \, \, \, \text{and}\, \, \, \, \, \psi(x,w,n)=0, \, \, \, \, \forall x,  z, w\in  A\cup B, m\in M, n\in N.
$$ 
For any $a_{1}\in \mathcal{A}, a_{2}\in A\cup B, m\in M$,
it follows from {\bf Claim 1} that
$$
\begin{aligned}
0=&\psi([e,a_{1}],a_{2},m)\\
=&[\psi(e, a_{2}, m),a_{1}]+[e,\psi(a_{1}, a_{2},m)]\\
=&\psi(e, a_{2}, m)a_{1}-a_{1}\psi(e, a_{2}, m)+e\psi(a_{1}, a_{2}, m)-\psi(a_{1}, a_{2}, m)e\\
=&-a_{1}\psi(e, a_{2}, m)+e\psi(a_{1}, a_{2}, m).
\end{aligned}
$$
This implies that 
$$
a_1\psi(e, a_2, m)=e\psi(a_{1}, a_{2}, m), \, \, \, \, \, \forall a_1\in \mathcal{A}, a_2\in A\cup B, m\in M.
\eqno{(3.39)}
$$
On the other hands, for all $a_2\in A, m\in M$, we have
$$
\begin{aligned}
0=e\psi(e, [e,a_{2}], m)f&=e[\psi(e, e, m),a_{2}]+[e,\psi(e, a_{2}, m)]f\\
=&\psi(e, e, m)a_{2}-a_{2}\psi(e, e, m)+e,\psi(e, a_{2}, m)-\psi(e, a_{2}, m)e\\
=&-a_{2}\psi(e, e, m)+e\psi(e, a_{2}, m).
\end{aligned}
$$
This gives rise to  
$$
a_2\psi(e, e, m)=e\psi(e, a_2, m), \, \, \, \, a_2\in A, m\in M.
\eqno{(3.40)}
$$
In view of the relations (3.39) and (3.40), we infer that for all $a_1\in A$
$$
\psi(a_{1}, a_{2}, m)=e\psi(a_{1}, a_{2}, m)f=\left\{
\begin{array}{rcl}
a_{1}a_{2}\psi(e, e, m)f=a_{2}a_{1}\psi(e, e, m)f& &{\text{if}~ a_{2}\in A,}\\
a_{1}\psi(e, e, m)a_{2}& &{\text{if}~ a_{2}\in B.}
\end{array} \right.\eqno{(3.41)}
$$
Using similar computational techniques and methods, for all $a_1\in A$, we arrive at
$$
\psi(a_{1}, a_{2}, n)=f\psi(a_{1}, a_{2}, n)e=\left\{
\begin{array}{rcl}
f\psi(e, e, n)a_{1}a_{2}=f\psi(e, e, n)a_{2}a_{1}& &{\text{if}~ a_{2}\in A,}\\
a_{2}\psi(e, e, n)a_{1}& &{\text{if}~ a_{2}\in B.}
\end{array} \right.\eqno{(3.42)}
$$

Considering the relations (3.41) and (3.42), we observe that for arbitrary elements
$x,y\in A\cup B$, the proofs of conclusions $\psi(x,y,m)=0$ and $\psi(x,y,n)=0$
can be reduced to those of $\psi(e,e,m)=0$ and $\psi(e,e,n)=0$ by using 
similar computational techniques and methods for all $m\in M, n\in N$.
In the subsequent, let us provide the proofs of
$\psi(e,e,m)=0$ and $\psi(e,e,n)=0$ for all $m\in M, n\in N$.

Let us construct the following two mappings
$$
\begin{aligned}
\xi: M & \longrightarrow M\\
m & \longmapsto \psi(e,e,m), \, \, \, \, \, \forall m\in M.
\end{aligned}
$$ 
and 
$$
\begin{aligned}
\zeta: N & \longrightarrow N\\
n & \longmapsto \psi(e,e,,n), \, \, \, \, \, \forall n\in N.
\end{aligned}
$$
For arbitrary elements $a\in A, b\in B$ and $m\in M$, we have
$$
\begin{aligned}
\xi(amb) &=\psi(e,e,amb)=e\psi(e,e,[a,mb])f\\
&=e([\psi(e,e,a),mb]+[a,\psi(e,e,mb)])f\\
&=a\psi(e,e,mb)f=a\psi(e,e,[m,b])f\\
&=a([\psi(e,e,m),b]+[m,\psi(e,e,b)])f\\
&=a\psi(e,e,m)b=a\xi(m)b.\\
\end{aligned}
$$
Consequently, the mapping $\xi: M\longrightarrow M$ is an
$(A,B)$-bimodule homomorphism. In an analogous way, one can show that 
the mapping $\zeta: N\longrightarrow N$ is a $(B,A)$-bimodule homomorphism.

According to the definition of generalized matrix algebras, we know that $mn\in A$ and $nm\in B$ 
holds true for all $m\in M, n\in N$. Combining this fact with Lemma \ref{xxsec3.10} yields
$$
\begin{aligned}
\mathcal{Z}(\mathcal{G}) \ni \psi(e,e,mn-nm)=&\psi(e,e,[m,n])\\
=&[\psi(e,e,m),n]+[m,\psi(e,e,n)]\\
=&\psi(e,e,m)n-n\psi(e,e,m)+m\psi(e,e,n)-\psi(e,e,n)m.
\end{aligned}
$$
With the help of above equation, we have
$$
e\psi(e,e,[m,n])e=e\psi(e,e,m)n+m\psi(e,e,n)e=\xi(m)n+m\zeta(n)\in \mathcal{Z}(A), \,  \, \forall m\in M, n\in N.    \eqno{(3.43)}
$$
and
$$
f\psi(e,e,[m,n])f=-n\psi(e,e,m)f-f\psi(e,e,n)m=-n\xi(m)-\zeta(n)m\in \mathcal{Z}(B),  \, \, \forall m\in M. n\in N.   \eqno{(3.44)}
$$
Substituting $am$ for $m$ in (3.43), we obtain 
$$
\mathcal{Z(G)}\ni e\psi(e,e,[am,n])e=\xi(am)n+am\zeta(n)=a\xi(m)n+am\zeta(n)=a\psi(e,e,[m,n])e.
$$
Replacing $n$ by $na$ in (3.44) give rises to 
$$
\mathcal{Z(G)}\ni e\psi(e,e,[m,na])e=\xi(m)na+m\zeta(na)=\xi(m)na+m\zeta(n)a=e\psi(e,e,[m,n])a.
$$
Thus the set $U=\{e\psi(e,e,[m,n])e \, \mid\, \forall m\in M, n\in N\}$ is a central ideal of algebra $A$.
According to the hypothesis (2) in Proposition \ref{xxsec3.1}, we might assume that 
the algebra $A$ contains no nonzero central ideals, 
In this case, the conclusion $U=0$ immediately follows, and consequently $\xi(m)n+m\zeta(n)=0$.
Applying the algebraic isomorphism $\eta:\pi_A(\mathcal{Z(G)})\longrightarrow
\pi_B(\mathcal{Z(G)})$ establishes the validity of equation
$n\xi(m)+\zeta(n)m=0$. Thus we assert that the mapping
$\xi:M \longrightarrow M$ and the mapping $\zeta: N\longrightarrow N$
constitute a special pair $(\xi, \zeta)$ of bimodule homomorphisms. Therefore,
by hypothesis (1) and (5) of Proposition \ref{xxsec3.1},
there exist central elements $a_0\in \mathcal{Z}(A)$
and $b_0\in \mathcal{Z}(B)$ satisfies equations
$$
\xi(m)=a_0m+mb_0=(a_0+\eta(b_0))m=\theta_0 m~\text{and}~\zeta(n)=-na_0-b_0n=-n(a_0+\eta(b_0))=-n\theta_0,
\eqno{(3.45)}
$$
where $\theta_0=a_0+\eta(b_0)\in \mathcal{Z}(A)=\pi_{A}(\mathcal{Z(G)})$ for all $m\in M, n\in N$.


We now proceed to prove that the central element $\theta_0\in \mathcal{Z(A)}$
possess property $\theta_0=0$. It should be noted that the proof will be carried out
by considering two distinct cases: either $MN\neq0$ or $NM\neq0$, and $MN=0=NM$.

Let us first consider the case of either $MN\neq0$ or $NM\neq0$. In this case, we have
$$
\begin{aligned}
\psi(e,n,m)=&\psi(e,[n,e],m)=[\psi(e,n,m),e]+[n,\psi(e,e,m)]\\
=&\psi(e,n,m)e-e\psi(e,n,m)+n\psi(e,e,m)-\psi(e,e,m)n
\end{aligned}\eqno{(3.46)}
$$
for all $m\in M, n\in N$. 
By multiplying both sides of equation (3.46) by either $e$ or $f$,
we conclude that 
$$
e\psi(e,n,m)e=-e\psi(e,e,m)n,~~f\psi(e,n,m)f=n\psi(e,e,m)f,~~e\psi(e,n,m)f=0
\eqno{(3.47)}
$$
for all $m\in M, n\in N$. On the other hand, 
$$
\begin{aligned}
\psi(e,n,m)=&\psi(e,n,[e,m])=[\psi(e,n,e),m]+[e,\psi(e,n,m)]\\
=&\psi(e,n,e)m-m\psi(e,n,e)+e\psi(e,n,m)-\psi(e,n,m)e.
\end{aligned}\eqno{(3.48)}
$$
holds true for all $m\in M, n\in N$. 
Premultiply the above equation by $f$ and postmultiply it by $e$ to get 
$2f\psi(e,n,m)e=0$. Since $\mathcal{R}$ is 2-torsion-free, we see that $f\psi(e,n,m)e=0$.
Furthermore, we have $\psi(e,n,m)\in  \left[
\begin{array}
[c]{cc}%
A & O\\
O & B\\
\end{array}
\right]$ 
for all $m\in M, n\in N$.

For all $m\in M, n,n_{1}\in N$, we have
$$
\begin{aligned}
\psi(n_{1},n,m)=&\psi([n_{1},e],n,m)\\
=&[\psi(n_{1},n,m),e]+[n_{1},\psi(e,n,m)]\\
=&\psi(n_{1},n,m)e-e\psi(n_{1},n,m)+n_{1}\psi(e,n,m)-\psi(e,n,m)n_{1}.
\end{aligned}\eqno{(3.49)}
$$
Multiplying the left side of (3.49) by $f$ and multiplying the right side of it by $e$ 
gives 
$$
n_1\psi(e,n,m)e=f\psi(e,n,m)n_1, \, \, \, \, \, \forall m\in M, n, n_1\in N.
\eqno{(3.50)}
$$
By rigorous verification and by employing analogous computational
techniques, we see that
$$
\begin{aligned}
&\psi(m_{1},n,m)\\
=&\psi([e,m_{1}],n,m)\\
=&[\psi(e,n,m),m_{1}]+[e,\psi(m_{1},n,m)]\\
=&\psi(e,n,m)m_{1}-m_{1}\psi(e,n,m)+e\psi(m_{1},n,m)-\psi(m_{1},n,m)e\\
\end{aligned}\eqno{(3.51)}
$$
for all $m, m_{1}\in M, n\in N$.
Premultiply (3.51) by $e$ and postmultiply (3.51) by $f$ to produce 
$$
e\psi(e,n,m)em_{1}=m_{1}f\psi(e,n,m)f, \, \, \, \, \, \forall m, m_1\in M, n\in N.
\eqno{(3.52)}
$$
Combining (3.50) with (3.52) together with the definition of the algebraic center $\mathcal{Z(G)}$,
we arrive at
$$
e\psi(e,n,m)e+f\psi(e,n,m)f\in\mathcal{Z(G)}, \, \, \, \, \, \forall m\in M, n\in N.           
\eqno{(3.53)}
$$

Making the best use of relations (3.45), (3.47) and (3.53) together with $\xi(m)=\theta_0m$, we see that
$$
e\psi(e,n,m)e=-\psi(e,e,m)n=-\theta_0 mn\in\mathcal{Z(G)}
$$
and 
$$
f\psi(e,n,m)f=n\psi(e,e,m)=n\theta_0 m=\eta(\theta_0)nm\in\mathcal{Z(G)}.
$$
Let us construct the set $Q=\{\, \theta_0 mn \mid \text{for all}\, m\in M, n\in N\}$,
which is clearly a central ideal of the algebra $A$. By the hypothesis (2) of Proposition \ref{xxsec3.1},
we conclude that $\theta_0 mn=0$ for all $m\in M, n\in N$. Furthermore, by the hypothesis $MN\neq \{0\}$,
it follows that there exist elements $m\in M$ and $n\in N$ such that $mn\neq0$.
Considering the hypothesis (3) of Proposition \ref{xxsec3.1} we assert that $\theta_0=0$.
Consequently, $\psi(e,e,m)=0$ and $\varphi(e,e,n)=0$ for all $m\in M, n\in N$. 
Using the relations (3.41) and (3.42), we conclude that 
$$
\psi(x,y,m)=0\, \, \,  \text{and}\, \, \, \varphi(x,z,n)=0
$$
for all $x, y, z\in A\cup B, y\in A\cup B, m\in M,  n\in N$.

Let us now turn to the case: $MN=\{0\}=NM$. According to the hypothesis (4) in Proposition \ref{xxsec3.1}, 
without loss of generality, we might assume that the algebra $A$ is noncommutative. Then there exist elements 
$a_1, a_2\in A$ such that $[a_1,a_2]\neq 0$.

Let us put $z=m$ in $\psi(x, y, z)$. So $\psi(x, y, m)$ is a $2$-Lie derivation (equivalently, Lie biderivation) with 
respect to the first and the second components. Using Lemma \ref{xxsec3.2} and taking $x=e, y=e, u=a_1, v=a_2$ into (LB) 
gives 
$$
[\psi(e,e,m),[a_{1},a_{2}]]+[[e,a_{2}],\psi(a_{1},e,m)]=[[e,e],\psi(a_{1},a_{2},m)]+[\psi(e,a_{2},m),[a_{1},e]] 
$$
for some $a_1, a_2\in A$ with $[a_1, a_2]\neq 0$ and for all $m\in M$. 
A straightforward computation produces 
$$
[\psi(e,e,m),[a_{1},a_{2}]]=0, \forall m\in M, \, \text{for some} \, a_1, a_2\in A \, \text{with} \, [a_1, a_2]\neq 0, 
$$
which is due to the fact $\psi(x, y, z)$ is a Lie biderivation with respect to the first and second components. 
With the help of $\psi(e,e,m)=\theta_0 m$, we have
$[\theta_0 m, [a_{1},a_{2}]]=0$ for all $m\in M$. This implies that $[a_{1},a_{2}]\theta_0 M=0$.
The faithful property of $(A, B)$-bimodule $M$ results in $[a_{1},a_{2}]\theta_0=0$. 
The assumption (3) of Proposition \ref{xxsec3.1} leads to $\theta_0=0$. That is, $\psi(e,e,m)=\psi(e,e,n)=0$.
Adopting similar computational methods as above, we get  
$\psi(m,e,e)=\psi(n,e,e)=0$ and $\psi(e,m,e)=\psi(e,n,e)=0$ for all $m\in M, n\in N$. 

Basing on the facts $\psi(e,e,m)=\psi(e,e,n)=0$ and utilizing the relations (3.41)-(3.42), we obtain
$$
\psi(x,y,m)=0~~ \text{and}~~\varphi(x,z,n)=0
$$
for all $x, y, z\in A\cup B,y\in A\cup B, m\in M, n\in N$.

\textbf{Claim 3.} With notations as above, we have
$$
\psi(x, y_1, m)=0~~ \text{and}~~\varphi(x, y_2, n)=0
$$
for all $x\in A\cup B, y_1, m \in M$ and $y_2,n\in N$.

In fact, for any $x\in  \left[
\begin{array}
[c]{cc}%
A & O\\
O & B\\
\end{array}
\right], y_1\in M $, we have
$$
\begin{aligned}
0=&\psi([e,x],y_{1},m)\\
=&[\psi(e,y_{1},m),x]+[e,\psi(x,y_{1},m)]\\
=&\psi(e,y_{1},m)x-x\psi(e,y_{1},m)+e\psi(x,y_{1},m)-\psi(x,y_{1},m)e.\\
\end{aligned}\eqno{(3.54)}
$$
According to the fact $\psi(x,m,m^\prime)\in M$ in {\bf Claim 1}, we know that
$$
\psi(x,y_{1},m)=e\psi(x,y_{1},m)f=\left\{
\begin{array}{rcl}
x\psi(e,y_{1},m)f& &{\text{whenever}~ x\in A,}\\
e\psi(e,y_{1},m)x& &{\text{whenever}~ x\in B}
\end{array} \right.             \eqno{(3.55)}
$$
for all $y_{1}, m\in M$. Using an analogous manner, one can establish the validity of
$$
\psi(x,y_{2},n)=f\psi(x,y_{2},n)e=\left\{
\begin{array}{rcl}
f\psi(e,y_{2}, n)x& &{\text{whenever}~ x\in A,}\\
x\psi(e,y_{2}, n)e& &{\text{whenever}~ x\in B}
\end{array} \right.      \eqno{(3.56)}
$$
for all $y_{2}, n\in N$.
In this sense, the proof of this claim is equivalent to verifying $\psi(e,y_1,m^\prime)=0$
and $\psi(e,y_2,n^\prime)=0$ for all $y_1, m^\prime \in M, y_2, n^\prime\in N$.

For any $y_1\in M$ and $y_2\in N$, we define establish the following two mappings 
$$
\begin{aligned}
\tau:M & \longrightarrow M\\
m & \longmapsto \psi(e, y_1, m), \, \, \, m\in M.
\end{aligned}
$$
and 
$$
\begin{aligned}
\sigma:N & \longrightarrow N\\
n & \longmapsto \psi(e, y_2, n), \, \, \, n\in N.
\end{aligned}
$$ 
Using an analogous method of the two mappings $\xi$ and $\zeta$ in {\bf Claim 2}, it is straightforward to check that the 
mapping $\tau:M \longrightarrow M$ is an $(A,B)$-bimodule homomorphism, while the mapping $\sigma:N \longrightarrow N$ is a
$(B,A)$-bimodule homomorphism.

For all $y_{1}\in M, n\in N$, by {\bf Claim 2} of this lemma, we see that
$$
\begin{aligned}
\psi(e,y_{1},n)=&\psi(e,[e,y_{1}],n)\\
=&[\psi(e,e,n),y_{1}]+[e,\psi(e,y_{1},n)]\\
=&e\psi(e,y_{1},n)-\psi(e,y_{1},n)e
\end{aligned}
$$
and
$$
\begin{aligned}
\psi(e,y_{1},n)=&\psi(e,y_{1},[n,e])\\
=&[\psi(e,y_{1},n),e]+[n,\psi(e,y_{1},e)]\\
=&\psi(e,y_{1},n)e-e\psi(e,y_{1},n).
\end{aligned}
$$
Adding the above two equations and using the 2-torsion-free property of
$\mathcal{R}$ results in $\psi(e,y_{1},n)=0$ for all $y_{1}\in M, n\in N$.
Adopting similar computational approaches yields $\psi(e,y_{2},m)=0$ for all
$y_{2}\in N, m\in M$.

Thus, for all $y_{1}, m\in M$ and $n\in N$, on the one hand,
$$
\begin{aligned}
0=&\psi(e,y_{1},mn)-\psi(e,y_{1},nm)\\
=&\psi(e,y_{1},[m,n])\\
=&[\psi(e,y_{1},m),n]+[m,\psi(e,y_{1},n)]\\
=&\psi(e,y_{1},m)n-n\psi(e,y_{1},m).
\end{aligned}
$$
And hence, $\tau(m)n=0$ and $n\tau(m)=0$ for all $m\in M$ and $n\in N$. 
For any $m\in M$, $y_2, n\in N$, on the other hand,
$$
\begin{aligned}
0=&\psi(e,y_{2},mn)-\psi(e,y_{2},nm)\\
=&\psi(e,y_{2},[m,n])\\
=&[\psi(e,y_{2},m),n]+[m,\psi(e,y_{2},n)]\\
=&m\psi(e,y_{2},n)-\psi(e,y_{2},n)m.
\end{aligned}
$$
So $m\sigma(n)=0$ and $\sigma(n)m=0$.
Consequently, we arrive at $m\sigma(n)+\tau(m)n=0$ and $n\tau(m)+\sigma(n)m=0$ for all $m\in M, n\in N$.
In particular, we see that $\tau$ and $\sigma$ are a special pair $(\tau, \sigma)$ of bimodule homomorphisms.
In view of the assumption (1) and (5) in Proposition \ref{xxsec3.1}, there exists $\alpha_{0}\in\pi_{A}(\mathcal{Z}(\mathcal{G}))$ 
such that $\tau(m)=\alpha_{0}m$ and $\sigma(n)=-n\alpha_{0}$.

Suppose first that $MN\neq0$ or $NM\neq0$.
Then, $\tau(m)n=\alpha_{0}mn=0$ for all $m\in M, n\in N$. With the help of assumption $(3)$ in Proposition \ref{xxsec3.1}, 
we get $\alpha_{0}=0$. That is, $\tau(m)=\alpha_{0}m=0$ and $\sigma(n)=-n\alpha_{0}=0$.
We therefore have $\psi(e,y_{1},m)=\psi(e,y_{2},n)=0$.

Let us deal with the case of $MN=0=NM$. Without loss of generality, we might assume that $A$ 
is a noncommutative algebra by invoking the assumption (4)  in Proposition \ref{xxsec3.1}.
Let $a_{1}, a_{2}\in A$ be fixed elements with $[a_{1}, a_{2}]\neq 0$.
Set $z=m$ in $\psi(x, y, z)$. Then $\psi(x, y, m)$ will be a $2$-Lie derivation (equivalently, Lie biderivation) with 
respect to the first and the second components. Using Lemma \ref{xxsec3.2} and taking $x=e, y=y_{1}, u=a_{1}, v=a_{2}$ into (LB) yields 
$$
[\psi(e,y_{1},m),[a_{1},a_{2}]]+[[e,a_{2}],\psi(a_{1},y_{1},m)]=[[e,y_{1}],\psi(a_{1},a_{2},m)]+[\psi(e,a_{2},m),[a_{1},y_{1}]].
$$
As in the end of the proof of {\bf Claim 2}, we observe that
$[\psi(e,y_{1},m),[a_{1},a_{2}]]=0$ for all $y_{1},m\in M$ and some
$a_{1}, a_{2}\in A$ with $[a_{1}, a_{2}]\neq 0$.
By the fact $\psi(e,y_{1},m)=\alpha_{0}m$ it follows that
$[\alpha_{0}m,[a_{1},a_{2}]]=0$.
A straightforward computation gives rise to $\alpha_{0}[a_{1},a_{2}]M=0$ for all $m\in M$ and
some $a_{1}, a_{2}\in A$ with $[a_{1}, a_{2}]\neq 0$. The faithfulness of the left $A$-module $M$
implies that $\alpha_{0}[a_{1},a_{2}]=0$. By the assumption (3) of Proposition \ref{xxsec3.1}, we assert that $\alpha_0=0$.
Thus $\tau(m)=\alpha_0 m=0$ and $\sigma(n)=-n \alpha_0=0$.
We therefore have $\psi(e,y_{1},m)=\psi(e,y_{2},n)=0$ for all $y_1, m \in M$ and $y_2, n\in N$.
In light of the equations (3.55) and (3.56), we obtain $\psi(x, y_1, m)=\psi(x, y_2, n)=0$
for all $x\in A\cup B, y_{1}, m\in M$ and $n, y_{2}\in N$.

Combining {\bf Claim 2} with {\bf Claim 3} of this Lemma, we conclude that $\psi(x, y_1, m)=\psi(x,y_2, n)=0$ 
for all $x\in A\cup B, y_1\in A\cup B\cup M, y_2\in A\cup B \cup N$ and $m\in M, n\in N$.

Employing similar computational techniques, we can reach the other conclusions of this lemma.

\end{proof}

\begin{tcolorbox}[breakable, enhanced, blanker, left=3mm,right=3mm,
borderline west={2pt}{0pt}{green}]

\begin{lemma}\label{xxsec3.12}
With notations as above, we have
$$
\psi(x, y, z)\in  \mathcal{Z(G)} ~\text{for~all}~ x, y, z\in A\cup B.
$$
\end{lemma}
\end{tcolorbox}

\begin{proof}
On the authority of  Lemma \ref{xxsec3.10}, let us first prove that $\psi(x,y,z)\in\mathcal{Z(G)}$ 
for all $x\in A$ and $y,z\in A \cup B $.

For all $a\in A, y,z\in A \cup B$ and $m\in M, n\in N$, with the help of Lemma \ref{xxsec3.11},
we obtain
$$
\begin{aligned}
0=&\psi(am,y,z)=\psi([a,m],y,z)\\
=&[\psi(a,y,z),m]+[a,\psi(m,y,z)]=[\psi(a,y,z),m]\\
=&\psi(a,y,z)m-m\psi(a,y,z)
\end{aligned}
$$
and
$$
\begin{aligned}
0=&\psi(na,y,z)=\psi([n,a],y,z)\\
=&[\psi(n,y,z),a]+[n,\psi(a,y,z)]=[n,\psi(a,y,z)]\\
=&n\psi(a,y,z)-\psi(a,y,z)n.
\end{aligned}
$$
Combining the above two equations yields
$$
\left[
\begin{array}
[c]{cc}%
\psi(a,y,z) & 0\\
0 & \psi(a,y,z)\
\end{array}
\right]\in \mathcal{Z(G)}        
\eqno{(3.57)}
$$
for all $a\in A, y,z \in A\cup B$. Similarly we also have 
$$
\left[
\begin{array}
[c]{cc}%
\psi(b,y,z) & 0 \\
0 & \psi(b,y,z) \\
\end{array}
\right]
\in \mathcal{Z(G)}
\eqno{(3.58)}
$$
for all $b\in B, y,z \in A\cup B$.
Considering the relations $(3.57)$ and $(3.58)$, we know that this lemma holds true.


\end{proof}

\begin{tcolorbox}[breakable, enhanced, blanker, left=3mm,right=3mm,
borderline west={2pt}{0pt}{green}]

\begin{lemma}\label{xxsec3.13}
With notations as above, we have
\begin{enumerate}
  \item [{\rm (i)}]  $\psi(x,n,m)=\psi(n,x,m)=\psi(x,m,n)=0;$
  \item [{\rm (ii)}] $\psi(x,m,n)=\psi(m,x,n)=\psi(x,n,m)=0;$
\end{enumerate}
 for all $x\in A\cup B$, $m\in M$, and $n\in N$.
\end{lemma}
\end{tcolorbox}

\begin{proof}
For all $x\in A\cup B, m\in M, n\in N$, it follows from Lemma \ref{xxsec3.11} that
$$
\begin{aligned}
\psi(x,n,m)=&\psi(x,[n,e],m)\\
=&[\psi(x, n, m),e]+[n,\psi(x, e, m)]\\
=&[\psi(x, n, m),e]\\
\end{aligned}
$$
and that
$$
\begin{aligned}
\psi(x,n,m)=&\psi(x,n,[e,m])\\
=&[\psi(x, n, e),m]+[e,\psi(x, n, m)]\\
=&[e,\psi(x, n, m)].
\end{aligned}
$$
Using the above two equations and considering the fact that $\mathcal{R}$ is $2$-torsionfree,
we see that 
$$
\psi(x,n,m)=0
$$ 
for all $x\in A\cup B, m\in M, n\in N$.

Employing analogous computational techniques and using Lemma \ref{xxsec3.11},
we establish the validity of the remaining cases of this lemma.
\end{proof}

\begin{tcolorbox}[breakable, enhanced, blanker, left=3mm,right=3mm,
borderline west={2pt}{0pt}{green}]

\begin{lemma}\label{xxsec3.14}
With notations as above,  we have
\begin{enumerate}
  \item [{\rm (i)}]  $\psi(m_{1},n,m_{2})=\psi(n,m_{1},m_{2})=\psi(m_{1},m_{2},n)=\psi(m_{2},n,m_{1})=\psi(n,m_{2},m_{1})=\psi(m_{2},m_{1},n)=0;$
  \item [{\rm (ii)}] $\psi(n_{1},m,n_{2})=\psi(m,n_{1},n_{2})=\psi(n_{1},n_{2},m)=\psi(n_{2},m,n_{1})=\psi(m,n_{2},n_{1})=\psi(n_{2},n_{1},m)=0;$
\end{enumerate}
for all  $m_{1}, m_{2}\in M$, and $n_{1}, n_{2}\in N$. 
\end{lemma}
\end{tcolorbox}

\begin{proof}
In light of Lemma \ref{xxsec3.13}, we see that
$$
\begin{aligned}
 \psi(m_{1},m_{2},n)=&\psi([e,m_{1}],m_{2},n)\\
=&[\psi(e,m_{2},n),m_{1}]+[e,\psi(m_{1},m_{2},n)]\\
=&[e,\psi(m_{1},m_{2},n)]\\
\end{aligned}
$$
and that
$$
\begin{aligned}
\psi(m_{1},m_{2},n)=&\psi(m_{1},m_{2},[n,e])\\
=&[\psi(m_{1},m_{2},n),e]+[n,\psi(m_{1},m_{2},e)]\\
=&[\psi(m_{1},m_{2},n),e].\\
\end{aligned}
$$
for all $m_{1}, m_{2}\in M~\text{and}~ n_{1}, n_{2}\in N$. 
Adding the above two equations together with $\mathcal{R}$ being $2$-torsionfree 
yields 
$$
\psi(m_{1},m_{2},n)=0
$$
for all $m_{1}, m_{2}\in M~\text{and}~ n\in N$.

Following an analogous approach and incorporating Lemma \ref{xxsec3.11},
one can achieve the remaining cases of the lemma.

\end{proof}

\begin{tcolorbox}[breakable, enhanced, blanker, left=3mm,right=3mm,
borderline west={2pt}{0pt}{green}]

\begin{lemma}\label{xxsec3.15}
With notations as above, we have
\begin{enumerate}
  \item [{\rm (i)}]  $\psi(m,m_{1},m_{2})=0;$
  \item [{\rm (ii)}] $\psi(n,n_{1},n_{2})=0;$
\end{enumerate}
for all  $m, m_{1}, m_{2}\in M$, and $n, n_{1}, n_{2}\in N$. 
\end{lemma}
\end{tcolorbox}

\begin{proof}
By invoking Lemmas \ref{xxsec3.11} and \ref{xxsec3.13}, we arrive at
$$
\begin{aligned}
\psi(m,m_{1},m_{2})=&\psi(m,[e,m_{1}],m_{2})\\
=&[\psi(m_{1},e,m_{2}),m_{1}]+[e,\psi(m,m_{1},m_{2})]\\
=&e\psi(m,m_{1},m_{2})-\psi(m,m_{1},m_{2})e\\
\end{aligned}
$$
and
$$
\begin{aligned}
\psi(n,n_{1},n_{2})=&\psi(n,[n_{1},e],n_{2})\\
=&[\psi(n,n_{1},n_{2}),e]+[n_{1},\psi(n,e,m_{2})]\\
=&\psi(n,n_{1},n_{2})e-e\psi(n,n_{1},n_{2}).
\end{aligned}
$$
for all  $m, m_{1}, m_{2}\in  M$ and $n, n_{1}, n_{2}\in N$. We therefore conclude 
$$
\psi(m,m_{1},m_{2})=e\psi(m,m_{1},m_{2})f\in M~ \text{and}~\psi(n,n_{1},n_{2})=f\psi(n,n_{1},n_{2})e\in N.
$$

To establish this result, let us choose two fixed elements $m_{1}, m_{2}\in M$. Then one can define two mappings 
$$
\begin{aligned}
\jmath: M & \longrightarrow  M\\
m & \longmapsto e\psi(m,m_{1},m_{2})f, \, \, \, \, \forall m\in M
\end{aligned}
$$ 
and 
$$
\begin{aligned}
\ell:N & \longrightarrow N\\
n & \longmapsto f\psi(n,n_{1},n_{2})e, \, \, \, \, \forall n\in N
\end{aligned}
$$ 
via 
$\jmath(m)=e\psi(m,m_{1},m_{2})f\in M$ and $\ell(n)=f\psi(n,n_{1},n_{2})e\in N$ respectively.


For all $a\in A, b\in B$ and $m\in M, n\in N$, it follows from Lemma \ref{xxsec3.11} that
$$
\begin{aligned}
\jmath(amb)=&e\psi(amb,m_{1},m_{2})f=e\psi([a,mb],m_{1},m_{2})f\\
=&e([\psi(a,m_{1},m_{2}),mb)+[a,\psi(mb,m_{1},m_{2})])f\\
=&ea\psi(mb,m_{1},m_{2})f=ea\psi([m,b],m_{1},m_{2})f\\
=&ea([\psi(m,m_{1},m_{2}),b]+[m,\psi(b,m_{1},m_{2})])f\\
=&ea([\psi(m,m_{1},m_{2}),b]f\\
=&a\psi(m,m_{1},m_{2})b=a\jmath(m)b
\end{aligned}
$$
and
$$
\begin{aligned}
\ell(bna)&=f\psi(bna,n_{1},n_{2})e=f\psi([b,na],n_{1},n_{2})e\\
&=f([b,\psi(na,n_{1},n_{2})]+[\psi(b,n_{1},n_{2}),na])e\\
&=fb\psi(na,n_{1},n_{2})e=fb\psi([n,a],n_{1},n_{2})e\\
&=fb([\psi(n,n_{1},n_{2}),a]+[n,\psi(a,n_{1},n_{2})])e\\
&=fb\psi(n,n_{1},n_{2})a.
\end{aligned}
$$
Thees imply that mapping
$\jmath: M\longrightarrow  M$ constitutes an
$(A,B)$-bimodule homomorphism, while the mapping $\ell:N\longrightarrow N$
is a $(B,A)$-bimodule homomorphism.

For all $m, m_1, m_2\in M$ and $n\in N$, in the light of Lemma \ref{xxsec3.14}, we obtain
$$
\begin{aligned}
0=&\psi(mn,m_{1},m_{2})-\psi(nm,m_{1},m_{2})\\
=&\psi([m,n],m_{1},m_{2})\\
=&[\psi(m,m_{1},m_{2}),n]+[m,\psi(n,m_{1},m_{2})]\\
=&[\jmath(m),n]\\
=&\jmath(m)n-n\jmath(m).
\end{aligned}
$$
Consequently, we get 
$$
\jmath(m)n=0~ \text{and}~ n\jmath(m)=0, \, \, \, \, \, \forall m\in M, n\in N. 
\eqno{(3.59)}
$$
Using an analogous computational approach together with Lemma \ref{xxsec3.14} gives 
$$
\begin{aligned}
0=&\psi(mn,n_{1},n_{2})-\psi(nm,n_{1},n_{2})\\
=&\psi([m,n],n_{1},n_{2})\\
=&[\psi(m,n_{1},n_{2}),n]+[m,\psi(n,n_{1},n_{2})]\\
=&[m,\ell(n)]\\
=&m\ell(n)-\ell(n)m
\end{aligned}
$$
for all $n, n_1, n_2\in N$ and $m\in M$. Hence
$$
m\ell(n)=0~\text{and}~\ell(n)m=0,   \, \, \, \, \, \forall m\in M, n\in N.      
\eqno{(3.60)}
$$ 
Basing on the two equations $\ell(n)m+n\jmath(m)=0$ and $m\ell(n)+\jmath(m)n=0$ and 
considering the bimodule homomorphism properties of mappings
$\jmath: M\longrightarrow  M$ and $\ell:N\longrightarrow N$,
we conclude that the mappings $\jmath: M\longrightarrow  M$
and $\ell:N\longrightarrow N$ form a special pair $(\jmath,\ell)$ of bimodule homomorphisms.
Consequently, there exist elements $\vartheta_0\in \mathcal{Z}(A)$ and $\vartheta_1\in \mathcal{Z}(B)$ such that
$\jmath(m)=\vartheta_0 m+m\vartheta_1$ and $\ell(n)=-n\vartheta_0 -\vartheta_1n$ for all $m\in M, n\in N$.
Furthermore, by virtue of the algebraic isomorphism
$\eta:\pi_{A}(\mathcal{Z}(\mathcal{G}))\longrightarrow \pi_{B}(\mathcal{Z}(\mathcal{G}))$,
we know that $\jmath(m)=\gamma  m$ and $\ell(n)=-n\gamma$ for all $m\in M, n\in N$,
where $\gamma=\vartheta_0+\eta(\vartheta_1)$.

Let us next show that $\gamma=0$, from which the relations $\psi(m,m_{1},m_{2})=e\psi(m,m_{1},m_{2})f=0$
and $\psi(n,n_{1},n_{2})=f\psi(n,n_{1},n_{2})e$ will follow.
To round off, let us divide the remaining discussion into two cases:
$MN\neq0$ or $NM\neq0$, and $MN=\{0\}=NM$.

Let us first deal with the case $MN\neq0$ or $NM\neq0$.
Taking into account (3.59), we get
$\gamma mn=\jmath(m)n=0=n\jmath(m)=n\gamma m=\eta(\gamma)nm$ for all $m\in M, n\in N$. 
In the case of $MN\neq 0$. Then there exist elements $m^\prime \in M, n^\prime \in N$ such that $m^\prime n^\prime \neq 0$. 
This results in the fact $\gamma m^\prime n^\prime =0$.
By the hypothesis (3) of Proposition \ref{xxsec3.1}, we conclude that $\gamma=0$. 
We therefore say that $\jmath(m)=\gamma m=0$ and $\ell(n)=-n\gamma=0$ for all $m\in M, n\in N$.
In this case, we eventually arrive at $\psi(m, m_{1}, m_{2})=\psi(n, n_{1}, n_{2})=0$
for all $m, m_{1}, m_{2}\in M$ and $n, n_{1}, n_{2}\in N$. Whenever $NM\neq 0$. 
Then there exist elements $n^{\prime\prime} \in N, m^{\prime\prime} \in M$ such that $n^{\prime\prime}, m^{\prime\prime} \neq 0$. 
Consequently, we see that $0=n^{\prime\prime} \jmath(m^{\prime\prime})=n^{\prime\prime}\gamma m^{\prime\prime}=\eta(\gamma)n^{\prime\prime}m^{\prime\prime}$. 
It follows from the hypothesis (3) of Proposition \ref{xxsec3.1} that $\eta(\gamma)=0$. 
And hence $\gamma=0$. This indicates that $\psi(m, m_{1}, m_{2})=\psi(n, n_{1}, n_{2})=0$ 
for all $m, m_{1}, m_{2}\in M$ and $n, n_{1}, n_{2}\in N$.

Our attention now turn to another case of $NM=0=NM$.
According to the hypothesis (4) of Proposition \ref{xxsec3.1}, and without loss of generality,
we may assume that the algebra $A$ is noncommutative. Consequently,
there exist elements $a_1, a_2\in A$ satisfying the relation $[a_1, a_2]\neq0$.
Set $z=m_2$ in $\psi(x, y, z)$. Then $\psi(x, y, m_2)$ will be a $2$-Lie derivation (equivalently, Lie biderivation) with 
respect to the first and the second components. Using Lemma \ref{xxsec3.2} and taking $x=m, y=m_{1}, u=a_{1}, v=a_{2}$ in the equation (LB) 
together with Lemma \ref{xxsec3.11} we get
$$
[\psi(m,m_{1},m_{2}),[a_{1},a_{2}]]+[[m,a_{2}],\psi(a_{1},m_{1},m_{2})]=[[m,m_{1}],\psi(a_{1},a_{2},m_{2})]+[\psi(m,a_{2},m_{2}),[a_{1},m_{1}]].
$$
By a direct computation of the above expression, we observe that
$$
[\psi(m,m_{1},m_{2}),[a_{1},a_{2}]]=0.
$$
In line with $\psi(m,m_{1},$ $m_{2})=\jmath(m)=\gamma m$, we obtain
$[\gamma m,[a_{1},a_{2}]]=0$. An immediate computation gives $\gamma [a_{1},a_{2}]m=0$ for all $m\in M$.
In view of the faithfulness of the left $A$-module $M$,
we see that $\gamma[a_{1},a_{2}]=0$. By the assumption (3) in Proposition \ref{xxsec3.1}, we assert that $\gamma=0$. 
That is, $\jmath(m)=\gamma m=0$ and $\ell(n)=-n\gamma=0$.
We therefore have $\psi(m,m_{1},m_{2})=\psi(n,n_{1},n_{2})=0$ for all $m, m_{1}, m_{2}\in M$
and $n, n_{1}, n_{2}\in N$.
\end{proof}

\begin{tcolorbox}[breakable, enhanced, blanker, left=3mm,right=3mm,
borderline west={2pt}{0pt}{brown}]

\noindent {\bf Proof of Proposition \ref{xxsec3.1}:}
\vspace{2mm}

Taking into accounts Lemmas \ref{xxsec3.10}-\ref{xxsec3.15}, we make sure that the well-established mapping $\psi$
possess the property 
$$
\psi(x, y, z)\in \mathcal{Z(G)}, \, \, \, \, \,  \forall x, y, z\in \mathcal{G}.
$$
Thus Proposition \ref{xxsec3.1} follows from Proposition \ref{xxsec3.7}. 
That is, every $3$-Lie derivation on a generalized matrix algebra $\mathcal{G}=\left[
\begin{array}
[c]{cc}%
A & M\\
N & B\
\end{array}
\right]$ can be decomposed into the
sum of an extremal $3$-derivation and a $3$-linear central-valued mapping.

Furthermore, note that the $3$-linear central-valued mapping $\psi$ is also a $3$-Lie derivation. 
Then for any $x, y, u, v \in G$, we have
$$
\psi([x,y],u,v)= [[\psi(x,u,v),y] + [x, \psi(y,u,v)] =0.
$$
This shows that $\psi$ vanishes on each commutator $[x, y]$ in the first component. 
Using analogous computational techniques, we see that the $3$-linear 
central-valued mapping $\psi$ also vanishes on arbitrary commutators in the second and third 
components. We therefore assert that $\psi$ vanishes on arbitrary commutators in each component.
\end{tcolorbox}

Applying Lemma \ref{xxsec3.4} and Proposition \ref{xxsec3.1} immediately yield the following corollary.

\begin{tcolorbox}[breakable, enhanced, blanker, left=3mm,right=3mm,
borderline west={2pt}{0pt}{blue}]

\begin{corollary}\label{xxsec3.16}
Let $\mathcal{G}=\left[
\begin{array}
[c]{cc}%
A & M\\
N & B\\
\end{array}
\right]$ be a generalized matrix algebra over a commutative ring $\mathcal{R}$ and
$\varphi:\mathcal{G}\times \mathcal{G}\times \mathcal{G}\longrightarrow\mathcal{G}$ be a $3$-Lie derivation on $\mathcal{G}.$ Suppose that
\begin{enumerate}
\item [{\rm (1)}] $\pi_{A}(\mathcal{Z}(\mathcal{G}))=\mathcal{Z}(A)$ and $\pi_{B}(\mathcal{Z}(\mathcal{G}))=\mathcal{Z}(B);$
\item [{\rm (2)}] either $A$ or $B$ does not contain nonzero central ideals;
\item [{\rm (3)}] if $\alpha a = 0, \alpha\in \mathcal{Z}(\mathcal{G}), 0 \neq a \in \mathcal{G},$ then $\alpha = 0;$
\item [{\rm (4)}] if $MN = 0 = NM$,
            then at least one of the algebras $A$ and $B$ is noncommutative;
\item [{\rm (5)}] every derivation on $\mathcal{G}$ is inner.
\end{enumerate}
Then $\varphi$ can be decomposed into $\varphi=\kappa+\psi$, where $\kappa$ is an extremal $3$-derivation such that 
$\kappa(x_1,x_2,x_3) =[x_1,[x_2,[x_3, X_0]]]$, $\psi$ is a $3$-linear central-valued mapping 
vanishing on commutators in each component, $X_0=e\varphi(e,e,e)f+(-1)^{3}f\varphi(e,e,e)e$.
\end{corollary}
\end{tcolorbox}

We now provide another main result of this section, which also characterize the structure
of $3$-Lie derivations on a generalized matrix algebra $\mathcal{G}=\left[
\begin{array}
[c]{cc}%
A & M\\
N & B\\
\end{array}
\right]$ from an alternative perspective.

\begin{tcolorbox}[breakable, blanker,left=3mm,right=3mm,
borderline west={2pt}{0pt}{orange}]

\begin{proposition}\label{xxsec3.17}
Let $\mathcal{G}=\left[
\begin{array}
[c]{cc}%
A & M\\
N & B\\
\end{array}
\right]$ be a generalized matrix algebra over a commutative ring $\mathcal{R}$
and $\varphi:\mathcal{G}\times \mathcal{G}\times \mathcal{G}\longrightarrow\mathcal{G}$
be a $3$-Lie derivation on $\mathcal{G}.$ Suppose that
\begin{enumerate}
\item [{\rm (1)}]$\pi_{A}(\mathcal{Z}(\mathcal{G}))=\mathcal{Z}(A)$ and $\pi_{B}(\mathcal{Z}(\mathcal{G}))=\mathcal{Z}(B)$;
\item [{\rm (2)}] Either $A$ or $B$ does not contain nonzero central ideals;
\item [{\rm (3)}] For each $m\in M$, the condition $mN=0=Nm$ implies $m=0$;
\item [{\rm (4)}] For each $n\in N$, the condition $Mn=0=nM$ implies $n=0$;
\item [{\rm (5)}] Each special pair of bimodule homomorphisms has standard form.
\end{enumerate}
Then $\varphi$ has the form $\varphi=\kappa+\psi$, 
where $\kappa$ is an extremal $3$-derivation such that $\kappa(x, y, z) =[x, [y, [z, X_0]]]$ for all $x, y, z\in \mathcal{G}$, 
$\psi$ is a $3$-linear central-valued mapping vanishing on commutators in each component, 
where $X_0= e\varphi(e, e, e)f+(-1)^3f\varphi(e, e, e)e$.
\end{proposition}
\end{tcolorbox}

We would like to point out that although this result also characterizes the structure of $3$-Lie derivations on
a generalized matrix algebra $\mathcal{G}=\left[
\begin{array}
[c]{cc}%
A & M\\
N & B\\
\end{array}
\right]$, its proof shares considerable similarities with the proof of Proposition \ref{xxsec3.1}.
Therefore, in what follows, we shall only present the distinct parts comparing with the proof of Proposition \ref{xxsec3.1}, 
while referring the reader to the proof of Proposition \ref{xxsec3.1} for the same parts.

It should be noted that the main differences between the two proofs originate from 
the distinctions between hypothesis (3) and (4) in Proposition \ref{xxsec3.17}
and their counterparts in Proposition \ref{xxsec3.1}. Consequently, in our subsequent proof,
the variations arising from these differences will be our intensive attentions.

\begin{tcolorbox}[breakable, blanker,left=3mm,right=3mm,
borderline west={2pt}{0pt}{green}]

\begin{lemma}\label{xxsec3.18}
With notations as above, for all $x\in A\cup B$, $m\in M$, and $n\in N$, we have
\begin{enumerate}
  \item[{\rm (i)}]  $\psi(x,y_{1},m)=\psi(y_{1},x,m)=\psi(x,m,y_{1})=\psi(y_{1},m,x)=\psi(m,x,y_{1})=\psi(m,y_{1},x)=0$
  for all $y_{1}\in A\cup M \cup B$;
  \item[{\rm (ii)}] $\psi(x,y_{2},n)=\psi(y_{2},x,n)=\psi(x,n,y_{2})=\psi(y_{2},n,x)=\psi(n,x,y_{2})=\psi(n,y_{2},x)=0$
  for all $y_{2}\in A\cup N \cup B$.
\end{enumerate}
\end{lemma}
\end{tcolorbox}

\begin{proof}
To prove this result, we will simultaneously establish $\psi(x,y_{1},m)=0$ and $\psi(x,y_{2},n)=0$
for all $x\in A\cup B, m\in M, n\in N$, $y_{1}\in A\cup M \cup B$ and $y_{2}\in A\cup N \cup B$.
For clarity of presentation, we divide our argument into
two cases basing on the value of $y$, {\bf Case 1}: $y_{1},y_{2} \in A\cup B$;
{\bf Case 2}: $y_{1} \in M, y_{2} \in N $.

{\bf Case 1}: $y_{1},y_{2} \in A\cup B$.

Note that these relations in this lemma are exactly those in Lemma \ref{xxsec3.11}.
Through a careful evaluation and verification, we find out that their proofs share significant similarities.
Adopting the same computational procedure as in Lemma \ref{xxsec3.11},
we can establish the mappings  $\xi: M \longrightarrow M$ and
$\zeta: N\longrightarrow N$ defined by $\xi(m)=\psi(e,e,m)\in M$ for all $m\in M$ and $\zeta(n)=\psi(e,e,n)\in N$ for all $n\in N$, respectively.
It is straightforward to check that the mapping $\xi: M \longrightarrow M$
is an $(A,B)$-bimodule homomorphism and that the mapping $\zeta: N\longrightarrow N$
is a $(B,A)$-bimodule homomorphism. 

Furthermore, the two bimodule homomorphisms
constitute a special pair $(\xi, \zeta)$ which enable the mappings $\xi: M \longrightarrow M$ and
$\zeta: N\longrightarrow N$ to satisfy the conditions 
$$
\xi(m)=a_0m+mb_0=(a_0+\eta(b_0))m=\theta_0 m~\text{and}~\zeta(n)=-na_0-b_0n=-n(a_0+\eta(b_0))=-n\theta_0,
$$
for some $a_0\in \mathcal{Z}(A)$ and $b_0\in \mathcal{Z}(B)$,
where $\theta_0=a_0+\eta(b_0)\in \pi_{A}(\mathcal{Z(G)})$ for all $m\in M, n\in N$.

Proceeding on our discussion as (3.46) and its subsequent proof,
we can establish
$$
e\psi(e,n,m)e+f\psi(e,n,m)f\in\mathcal{Z(G)}, \, \, \, \, \, \forall m\in M, n\in N, 
\eqno{(3.61)}
$$
which is identical with the relation (3.53). 
Using the fact that
the second component is a Lie derivation together with (3.61), we get
$$
\begin{aligned}
\psi(e,n,m)&=\psi(e,[n,e],m)\\
&=[\psi(e,n,m),e]+[n,\psi(e,e,m)]\\
&=\psi(e,n,m)e-e\psi(e,n,m)+n\psi(e,e,m)-\psi(e,e,m)n\\
\end{aligned} 
\eqno{(3.62)}
$$
for all $m\in M, n\in N$.
Premultiply and postmultiply (3.62) by $e$ to obtain 
$$
e\psi(e,n,m)e=-\psi(e,e,m)n=-\theta_0 mn\in\mathcal{Z(G)}, \, \, \, \, \, \forall m\in M, n\in N.
$$
Multiply (3.62) by $f$ from two sides to get 
$$
f\psi(e,n,m)f=n\psi(e,e,m)f=n\theta_0 m=\eta(\theta_0)nm\in\mathcal{Z(G)}, \, \, \, \, \, \forall m\in M, n\in N.
$$
Define the set $Q=\{ \, \theta_0 mn \mid \, \text{for all}\, m\in M, n\in N\}$. It is clear that $Q$ is a 
central ideal of the algebra $A$. According to the hypothesis (2) of Proposition \ref{xxsec3.17}, without loss of generality, 
we might assume that the algebra $A$ does not contain nonzero central ideals. 
This gives that $Q=0$. That is $(\theta_0 m)n=0$ for all $n\in N$.
It follows from (3.61) that
$$
n\theta_0 m=f\psi(e,n,m)f=\eta(e\psi(e,n,m)e)=0, \, \, \, \, \, \forall m\in M, n\in N.
$$
And then $N(\theta_0 m)=0$. Combining the fact $N(\theta_0 m)=0=(\theta_0 m)N$ with the 
hypothesis $(3)$ of Proposition \ref{xxsec3.17} gives
$\theta_0 m=0$. The faithfulness of $M$ implies that $\theta_0=0$. According to the definitions of mappings $\xi: M \longrightarrow M$ and
$\zeta: N\longrightarrow N$, it can be seen that the relations $\xi(m)=\psi(e,e,m)=0$
and $\zeta(n)=\psi(e,e,n)=0$ hold true for all $m\in M, n\in N$.

Finally, according to the relations (3.41) and (3.42), we conclude that the mappings
$\psi(x,y_{1},m)=0$ and $\psi(x,y_{2},n)=0$ for all $x\in A\cup B, y_1\in A\cup B,
y_2\in A\cup B$ and $m\in M, n\in N$.
The remaining conclusions in (i) and (ii) of this lemma can be achieved in an analogous manner.

{\bf Case 2}: Let us next consider the second case $y_1\in M, y_2\in N$.

For this case, let us adhere to the proof strategy of {\bf Claim 3} of Lemma \ref{xxsec3.11}.
Through careful evaluations and detailed computations,
we make sure that the relations (3.55) and (3.56) still hold true. Consequently, the mappings $\tau:M\longrightarrow M$ and $\sigma:N\longrightarrow N$,
defined by $\tau(m)=\psi(e,y_1, m)$ and $\sigma(n)=\psi(e, y_2, n)$
for all $y_1,m\in M$ and $y_2, n\in N$ respectively, are also reasonable.
Employing the same methods and techniques, one can show that the mapping $\tau:M\longrightarrow M$
is an $(A,B)$-bimodule homomorphism and that the mapping $\sigma:N\longrightarrow N$ is a
$(B,A)$-bimodule homomorphism. In this sense, the mappings $\tau:M\longrightarrow M$ and $\sigma:N\longrightarrow N$
form special pairs $(\tau, \sigma)$ of bimodule homomorphisms.

By the hypothesis (1) and (5) of Proposition \ref{xxsec3.17}, there exists a central
element $\alpha_0\in \mathcal{Z}(A)$ such that
$$
\tau(m)=\alpha_0 m~~\text{and}~~\sigma(n)=-n \alpha_0
$$
for all $m\in M, n\in N$. In the process of verifying that mappings $\tau:M\longrightarrow M$ and $\sigma:N\longrightarrow N$
are special pairs $(\tau, \sigma)$, we further establish the equations $n\tau(m)=0$ and $m\sigma(n)=0$,
from which equations $n(\alpha_0 m)=0$ and $m(n \alpha_0)=(\alpha_0m)n=0$ follow.
Considering the hypothesis (3) of Proposition \ref{xxsec3.17}, we know that $\alpha_0m=0$ for all $m\in M$.
The faithfulness of $M$ leads to $\alpha_0=0$.
Thus $\tau(m)=0$ and $\sigma(n)=0$ hold true for all $m\in M, n\in N$. This is equivalent to saying that 
$\psi(e,y_1, m)=0$ and $\psi(e, y_2, n)=0$ for all $y_1, m\in M, y_2, n\in N$.
Applying the equations (3.55) and (3.56) yields $\psi(x,y_1, m)=0$ and $\psi(x, y_2, n)=0$ for all $x\in A\cup B, y_1, m\in M$ and $y_2, n\in N$.

The remaining statements (i) and (ii) in this lemma can be obtained in an analogous method and strategy.

\end{proof}

A thorough analysis and meticulous computation of Lemmas \ref{xxsec3.12}-\ref{xxsec3.14}
reveal that their proofs rely solely on:
the behavior of Lie derivations with respect to each component, the characterization of the algebraic center, and
conclusions obtained from previous lemmas.
Consequently, those fundamental and underlying methods employed in Lemmas \ref{xxsec3.12}-\ref{xxsec3.14} 
could be moved to here under the new perspective without any essential modifications. 
Of course, the corresponding results and counterparts can be successfully achieved.
To facilitate the exposition of Proposition \ref{xxsec3.17}, we are going to restate these conclusions without detailed proofs.

\begin{tcolorbox}[breakable, blanker,left=3mm,right=3mm,
borderline west={2pt}{0pt}{green}]

\begin{lemma}\label{xxsec3.19}
With notations as above, we have
$$
\psi(x,y,z)\in  \mathcal{Z(G)}, \, \, \, \, \, \forall x, y, z\in A\cup B.
$$
\end{lemma}
\end{tcolorbox}

\begin{tcolorbox}[breakable, blanker,left=3mm,right=3mm,
borderline west={2pt}{0pt}{green}]

\begin{lemma}\label{xxsec3.20}
With notations as above, for all $x\in A\cup B$, $m\in M$, and $n\in N$, we have
\begin{enumerate}
  \item [{\rm (i)}]  $\psi(x,n,m)=\psi(n,x,m)=\psi(x,m,n)=0$;
  \item [{\rm (ii)}] $\psi(x,m,n)=\psi(m,x,n)=\psi(x,n,m)=0$.
\end{enumerate}
\end{lemma}
\end{tcolorbox}

\begin{tcolorbox}[breakable, blanker,left=3mm,right=3mm,
borderline west={2pt}{0pt}{green}]

\begin{lemma}\label{xxsec3.21}
With notations as above, for all  $m_{1}, m_{2}\in M$, and $n_{1}, n_{2}\in N$, we have
\begin{enumerate}
  \item[{\rm (i)}]  $\psi(m_{1},n,m_{2})=\psi(n,m_{1},m_{2})=\psi(m_{1},m_{2},n)=\psi(m_{2},n,m_{1})=\psi(n,m_{2},m_{1})=\psi(m_{2},m_{1},n)=0$;
  \item[{\rm (ii)}] $\psi(n_{1},m,n_{2})=\psi(m,n_{1},n_{2})=\psi(n_{1},n_{2},m)=\psi(n_{2},m,n_{1})=\psi(m,n_{2},n_{1})=\psi(n_{2},n_{1},m)=0$.
\end{enumerate}
\end{lemma}
\end{tcolorbox}

\begin{tcolorbox}[breakable, blanker,left=3mm,right=3mm,
borderline west={2pt}{0pt}{green}]

\begin{lemma}\label{xxsec3.22}
With notations as above, for all  $m, m_{1}, m_{2}\in M$, and $n, n_{1}, n_{2}\in N$, we have
\begin{enumerate}
  \item[{\rm (i)}]  $\psi(m,m_{1},m_{2})=0;$
  \item[{\rm (ii)}] $\psi(n,n_{1},n_{2})=0$.
\end{enumerate}
\end{lemma}
\end{tcolorbox}

\begin{proof}
Using an analogous method as in Lemma \ref{xxsec3.15}, we have 
$\psi(m,m_{1},m_{2})\in M$ and $\psi(n,n_{1},n_{2})\in N$ for all $m, m_1, m_2\in M, n, n_1, n_2\in N$.
Applying Lemmas \ref{xxsec3.18}-\ref{xxsec3.21}, we arrive at 
$$
\begin{aligned}
0=&\psi(mn, m_{1}, m_{2})-\psi(nm, m_{1}, m_{2})\\
=&\psi([m, n], m_{1}, m_{2})\\
=&[\psi(m, m_{1}, m_{2}), n]+[m,\psi(n, m_{1}, m_{2})]\\
=&[\psi(m, m_{1}, m_{2}), n]\\
=&\psi(m, m_{1}, m_{2})n-n\psi(m, m_{1}, m_{2})
\end{aligned}
$$
for all $m, m_1, m_2\in M, n\in N$. This implies that $\psi(m, m_{1}, m_{2})n=0$ and $n\psi(m, m_{1}, m_{2})=0$.
The conclusion $\psi(m, m_{1}, m_{2})=0$ follows from the hypothesis (3) in Proposition \ref{xxsec3.17}.

Adopting similar computational approaches and utilizing the hypothesis (4) of Proposition \ref{xxsec3.17},
we have $\psi(n, n_{1}, n_{2})=0$ for all $n, n_1, n_2\in N$.
\end{proof}

\begin{tcolorbox}[breakable, blanker,left=3mm,right=3mm,
borderline west={2pt}{0pt}{brown}]

\noindent {\bf Proof of Proposition \ref{xxsec3.17}:}
\vspace{2mm}

Through a careful computation, it is straightforward to see that Proposition \ref{xxsec3.7} 
holds true under the assumptions of Proposition \ref{xxsec3.17}. By invoking Proposition \ref{xxsec3.7}, 
it follows that $\varphi$ can be decomposed into 
the sum of an extremal $3$-derivation $\kappa$ and a Lie $3$-derivation $\psi$ satisfying the condition 
$\psi(e,e,e)\in \mathcal{Z}(\mathcal{G})$. That is, $\varphi=\kappa+\psi$. In view of Lemma 
\ref{xxsec3.10} and Lemmas \ref{xxsec3.18}-\ref{xxsec3.22}, 
we conclude that the well-established mapping $\psi$ possess the property 
$$
\psi(x, y, z)\in \mathcal{Z(G)}, \, \, \, \, \,  \forall x, y, z\in \mathcal{G}.
$$
Therefore, under the assumptions of Proposition \ref{xxsec3.17}, every $3$-Lie derivation on a generalized matrix algebra $\mathcal{G}=\left[
\begin{array}
[c]{cc}%
A & M\\
N & B\\
\end{array}
\right]$ can 
can be expressed as the sum of an extremal $3$-derivation and a $3$-linear central-valued mapping.

In addition, since the $3$-linear central-valued mapping $\psi$ is also a $3$-Lie derivation, we know that
$$
\psi([x,y],u,v)= [[\psi(x,u,v),y] + [x, \psi(y,u,v)] =0
$$
for all $x, y, u, v \in G$. This implies that $\psi$ vanishes on an arbitrary commutator $[x, y]$ in the first component. 
Similarly, the $3$-linear central-valued mapping $\psi$ also vanishes on all commutators in the second and third 
components. So $\psi$ vanishes on all commutators in each component.
\end{tcolorbox}

As an immediate consequence of Lemma \ref{xxsec3.4} and Proposition \ref{xxsec3.17},
we have the following corollary:

\begin{tcolorbox}[breakable, blanker,left=3mm,right=3mm,
borderline west={2pt}{0pt}{blue}]

\begin{corollary}\label{xxsec3.23}
Let $\mathcal{G}=\left[
\begin{array}
[c]{cc}%
A & M\\
N & B\\
\end{array}
\right]$ be a generalized matrix algebra over a commutative ring $\mathcal{R}$
and $\varphi:\mathcal{G}\times \mathcal{G}\times \mathcal{G}\longrightarrow\mathcal{G}$
be a $3$-Lie derivation on $\mathcal{G}$. Suppose that
\begin{enumerate}
\item [{\rm (1)}] $\pi_{A}(\mathcal{Z}(\mathcal{G}))=\mathcal{Z}(A)$ and $\pi_{B}(\mathcal{Z}(\mathcal{G}))=\mathcal{Z}(B);$
\item [{\rm (2)}] either $A$ or $B$ does not contain nonzero central ideals;
\item [{\rm (3)}] For each $n\in N$, the condition $Mn = 0 = nM$ implies $n=0$;
\item [{\rm (4)}] For each $m\in M$, the condition $Nm = 0 = mN$ implies $m=0$;
\item [{\rm (5)}] each derivation on $\mathcal{G}$ is inner.
\end{enumerate}
Then $\varphi$ is of the form $\varphi=\kappa+\psi,$
where $\kappa$ is an extremal $3$-derivation such that
$\kappa(x, y, z) =[x, [y, [z, X_0]]]$ for all 
$x, y, z\in \mathcal{G}$, $\psi$ is a $3$-linear central-valued mapping vanishing on commutators 
in each component, $X_0= e\varphi(e, e, e)f+(-1)^3f\varphi(e, e, e)e$.
\end{corollary}
\end{tcolorbox}

\section{Main Theorems and Its Applications}
\label{xxsec4}

This section is devoted to our two main theorems and its detailed proofs. Some related applications are also presented.

\begin{tcolorbox}[breakable, blanker,left=3mm,right=3mm,
borderline west={2pt}{0pt}{red}]

\begin{theorem}\label{xxsec4.1}
Let $\mathcal{G}=\left[
\begin{array}
[c]{cc}%
A & M\\
N & B\\
\end{array}
\right]$ be a generalized matrix algebrabe over a commutative ring $\mathcal{R}$ and
$\varphi:\underbrace{\mathcal{G}\times \cdots \times \mathcal{G}}_n\longrightarrow\mathcal{G}$ be a
$n$-Lie derivation on $\mathcal{G}\, \, (n\geq 3)$ Suppose that $\mathcal{G}$ satisfies the following conditions:
\begin{enumerate}
\item [{\rm (1)}] $\pi_{A}(\mathcal{Z}(\mathcal{G}))=\mathcal{Z}(A)$ and $\pi_{B}(\mathcal{Z}(\mathcal{G}))=\mathcal{Z}(B);$
\item [{\rm (2)}] either $A$ or $B$ does not contain nonzero central ideals;
\item [{\rm (3)}] if $\alpha a = 0, \alpha\in \mathcal{Z}(\mathcal{G}), 0 \neq a \in \mathcal{G}$, then $\alpha = 0;$
\item [{\rm (4)}] If $MN = 0 = NM$,
            then at least one of the algebras $A$ and $B$ is noncommutative;
\item [{\rm (5)}] every special pair of bimodule homomorphisms has standard form.
\end{enumerate}
Then $\varphi$ is of the form $\varphi=\kappa+\psi$,
where $\kappa$ is an extremal $n$-derivation such that
$\kappa(x_1,x_2,\cdots,x_n) =[x_1,[x_2,[\cdots,[x_n, X_0]\cdots]]]$ for all $x_1, x_2, \cdots, x_n\in \mathcal{G}$, 
$\psi$ is an $n$-linear central-valued mapping vanishing on commutators in each component, 
$X_0=e\varphi(e,e,\cdots,e)f+(-1)^{n}f\varphi(e,e,\cdots,e)e$. 
\end{theorem}
\end{tcolorbox}

\begin{proof}
Let us take induction method for the multiplicity $n$. The initial case of $n=3$ is
verified by Proposition \ref{xxsec3.1}. Suppose that this result is valid for all cases with multiplicities $\leq n-1$. 
We will complete this proof of this result by showing that it holds true for the case of multiplicity $n$.

For fixed elements $x_4,\cdots,x_n\in{\mathcal{G}}$, let us define a $3$-linear mapping 
$$
\begin{aligned}
\chi_{x_4,\cdots,x_n}:\mathcal{G}\times\mathcal{G}\times\mathcal{G} & \longrightarrow \mathcal{G}\\
(x_1,x_2,x_3) & \longmapsto \varphi(x_1,x_2,x_3,\cdots,x_n), \, \, \, \forall x_1, x_2, x_3\in \mathcal{G}.
\end{aligned}
$$
It is clear that $\chi_{x_4,\cdots,x_n}(x_1,x_2,x_3)$ is a $3$-Lie derivation. Then for any 
$x_1 x_2, x_3\in \mathcal{G}$, applying Proposition \ref{xxsec3.1} yields 
$$
\begin{aligned}
\varphi(x_1,x_2,x_3,\cdots,x_n) & =\chi_{x_4,\cdots,x_n}(x_1,x_2,x_3)\\
&=[x_1,[x_2,[x_3,e\chi_{x_4,\cdots,x_n}(e,e,e)f+(-1)^{3}f\chi_{x_4,\cdots,x_n}(e,e,e)e]]]\\
&\, \, \, \, \, \, \, \, +\psi_{x_4,\cdots,x_n}(x_1,x_2,x_3)\\
&= [x_1,[x_2,[x_3,e\varphi(e,e,e,x_4,\cdots,x_n)f+(-1)^{3}f\varphi(e,e,e,x_4,\cdots,x_n)e]]]\\
& \, \, \, \, \, \, \, \, +\psi_{x_4,\cdots,x_n}(x_1,x_2,x_3).
\end{aligned}
$$
It should be remarked that $\varphi(x_1, x_2, x_3, \cdots, x_n)$ and 
$[x_1,[x_2,[x_3,e\varphi(e,e,e,x_4,\cdots,x_n)f+(-1)^{3}f\varphi(e,e,e,x_4, $ $\cdots,x_n)e]]]$ 
are both $n$-linear mappings. Let us define their difference to be $\psi(x_1, x_2, x_3, \cdots, x_n)$, which 
is an $n$-linear mapping. That is 
$$
\begin{aligned}
\varphi(x_1,x_2,x_3,\cdots,x_n)&= \chi_{x_4,\cdots,x_n}(x_1,x_2,x_3)\\
&= [x_1,[x_2,[x_3,e\chi_{x_4,\cdots,x_n}(e,e,e)f+(-1)^{3}f\chi_{x_4,\cdots,x_n}(e,e,e)e]]]\\
&\, \, \, \, \, \, \, \, +\psi_{x_4,\cdots,x_n}(x_1,x_2,x_3)\\
& \eqdef [x_1,[x_2,[x_3,e\varphi(e,e,e,x_4,\cdots,x_n)f+(-1)^{3}f\varphi(e,e,e,x_4,\cdots,x_n)e]]]\\
& \, \, \, \, \, \, \, +\psi(x_1,x_2,x_3,\cdots,x_n), 
\end{aligned}
$$
where $\psi_{x_4,\cdots,x_n}(x_1,x_2,x_3)$ is a $3$-linear central-valued mapping
and $\psi_{x_4,\cdots,x_n}(x_1,x_2,x_3)=\psi(x_1,x_2,x_3,x_4, $ $\cdots,x_n)$.

In addition, by induction assumption for $n-1$, it is clear that
$\varphi(e,x_2,x_3,\cdots,x_n)$ is an $(n-1)$-Lie derivation for all $x_2,\cdots,x_n\in{\mathcal{G}}$. This implies that
$$
\begin{aligned}
\varphi(e,x_2,x_3,\cdots,x_n)=&[x_2,[x_3,\cdots,[x_n,e\varphi(e,\cdots,e)f+(-1)^{n-1}f\varphi(e,\cdots,e)e]\cdots]]\\
&+\psi(e,x_2,x_3,x_4,\cdots,x_n)
\end{aligned}
$$
for all $x_2,\cdots,x_n\in{\mathcal{G}}$. In particular,
$$
\begin{aligned}
\varphi(e,e,e,x_4,\cdots,x_n)=&[e,[e,[x_4\cdots,[x_n,e\varphi(e,\cdots,e)f+(-1)^{n-1}f\varphi(e,\cdots,e)e]\cdots]]]\\
&+\psi(e,e,e,x_4,\cdots,x_n)\\
=&[e,[e,[x_4,[x_5,\cdots,[x_n,e\varphi(e,\cdots,e)f]\cdots]]]]\\
&+[e,[e,[x_4,[x_5,\cdots,[x_n,(-1)^{n-1}f\varphi(e,\cdots,e)e]\cdots]]]]\\
&+\psi(e,e,e,x_4,\cdots,x_n)\\
\stackrel{\Delta}{=}&[x_4,[x_5,\cdots,[x_n,e\varphi(e,\cdots,e)f]\cdots]]\\
&+[x_4,[x_5,\cdots,[x_n,(-1)^{n-1}f\varphi(e,\cdots,e)e]\cdots]]\\
&+\psi(e,e,e,x_4,\cdots,x_n)\\
=&[x_4,[x_5,\cdots,[x_n,e\varphi(e,\cdots,e)f+(-1)^{n-1}f\varphi(e,\cdots,e)e]\cdots]]\\
&+\psi(e,e,e,x_4,\cdots,x_n)
\end{aligned}\eqno{(4.1)}
$$
for all $x_4,\cdots,x_n\in{\mathcal{G}}$. Here, the verification for the equality $\stackrel{\Delta}{=}$ is due to the following fact:
$$
[x,m_0]=e[x,m_0]f ~\text{and}~  [x,n_0]=f[x,n_0]e, \, \, \, \, \, \, \forall x\in \mathcal{G},                           \eqno{(4.2)}
$$
where $m_0=e\varphi(e, \cdots, e)f\in M, n_0=f\varphi(e, \cdots, e)e\in N$.
Indeed, a careful computation shows that the elements $m_0$ and $ n_0$ satisfy the conclusion of {\bf Claim 4} 
in Proposition \ref{xxsec3.6}. That is, $m_0N = Nm_0 = 0$ and $n_0M = Mn_0 = 0$.  Then for any $x=\left[
\begin{array}
[c]{cc}%
a & m\\
n & b\\
\end{array}
\right]\in \mathcal{G}$, we see that
$$
\begin{aligned}\left[\left[
\begin{array}
[c]{cc}%
a & m\\
n & b\\
\end{array}
\right],\left[
\begin{array}
[c]{cc}%
0 & 0\\
n_0 & 0\\
\end{array}
\right]\right]=&\left[
\begin{array}
[c]{cc}%
mn_0 & 0\\
n_0a-bn_0 & n_0m\\
\end{array}
\right]=\left[
\begin{array}
[c]{cc}%
0 & 0\\
n_0a-bn_0 & 0\\
\end{array}
\right]\\
=&f\left(\left[\left[
\begin{array}
[c]{cc}%
a & m\\
n & b\\
\end{array}
\right],\left[
\begin{array}
[c]{cc}%
0 & 0\\
n_0 & 0\\
\end{array}
\right]\right]\right)e
\end{aligned}$$
and that 
$$
\begin{aligned}
\left[\left[
\begin{array}
[c]{cc}%
a & m\\
n & b\\
\end{array}
\right],\left[
\begin{array}
[c]{cc}%
0 & m_0\\
0 & 0\\
\end{array}
\right]\right]=&\left[
\begin{array}
[c]{cc}%
m_0n & am_0-m_0b\\
0 & nm_0\\
\end{array}
\right]=\left[
\begin{array}
[c]{cc}%
0 & am_0-m_0b\\
0 & 0\\
\end{array}
\right]\\=&e\left(\left[\left[
\begin{array}
[c]{cc}%
a & m\\
n & b\\
\end{array}
\right],\left[
\begin{array}
[c]{cc}%
0 & m_0\\
0 & 0\\
\end{array}
\right]\right]\right)f.
\end{aligned}
$$
It follows from (4.2) that 
$$
e\varphi(e,e,e,x_4,\cdots,x_n)f=[x_4,[x_5,\cdots,[x_n,e\varphi(e,\cdots,e)f]\cdots]]         
\eqno{(4.3)}
$$
and
$$
f\varphi(e,e,e,x_4,\cdots,x_n)e=[x_4,[x_5,\cdots,[x_n,(-1)^{n-1}f\varphi(e,\cdots,e)e]\cdots]]    
\eqno{(4.4)}
$$
In light of the relations (4.3) and (4.4), we have
$$
\begin{aligned}
\varphi(x_1,x_2,x_3,\cdots,x_n)=& [x_1,[x_2,[x_3,e\varphi(e,e,e,x_4,\cdots,x_n)f+(-1)^{3}f\varphi(e,e,e,x_4,\cdots,x_n)e]]]\\
&+\psi(x_1,x_2,x_3,x_4,\cdots,x_n)\\
=&[x_1,[x_2,[x_3,e\varphi(e,e,e,x_4,\cdots,x_n)f]]]\\
 &+[x_1,[x_2,[x_3,(-1)^{3}f\varphi(e,e,e,x_4,\cdots,x_n)e]]]\\
&+\psi(x_1,x_2,x_3,x_4,\cdots,x_n)\\
=&[x_1,[x_2,[x_3,[x_4,[x_5,\cdots,[x_n,e\varphi(e,\cdots,e)f]\cdots]]]]]\\
 &+[x_1,[x_2,[x_3,[x_4,[x_5,\cdots,[x_n,(-1)^{n}f\varphi(e,\cdots,e)e]\cdots]]]]]\\
&+\psi(x_1,x_2,x_3,x_4,\cdots,x_n)\\
=&[x_1,[x_2,[x_3,[x_4,[x_5,\cdots,[x_n,e\varphi(e,\cdots,e)f+(-1)^{n}f\varphi(e,\cdots,e)e]\cdots]\cdots]]]]]\\
&+\psi(x_1,x_2,x_3,x_4,\cdots,x_n)\\
\end{aligned}\eqno{(4.5)}
$$
for all $x_1,\cdots,x_n\in \mathcal{G}$. This implies that
$$
\varphi(x_1,x_2,\cdots,x_n)=[x_1,[x_2,\cdots,[x_n,e\varphi(e,\cdots,e)f+(-1)^{n}f\varphi(e,\cdots,e)e]\cdots]]+\psi(x_1,x_2, \cdots,x_n)
$$
 for all $x_1,\cdots,x_n\in{\mathcal{G}}$. 
 Obviously, the elements $m_0$ and $n_0$ satisfy the conclusion of {\bf Claim 4} in Proposition \ref{xxsec3.7}. 
 Considering Proposition \ref{xxsec3.5}, we know that $e\varphi(e,\cdots,e)f+(-1)^{n}f\varphi(e,\cdots,e)e$ satisfies the equation 
 $$
 [[\mathcal{G}, \mathcal{G}], e\varphi(e,\cdots,e)f+(-1)^{n}f\varphi(e,\cdots,e)e]=0.
 $$
 Furthermore, with help of Remark \ref{xxsec3.6} we see that 
$[x_1, [x_2, \cdots,  [x_n, e\varphi(e, \cdots, e)f+(-1)^n f\varphi(e,\cdots, e)e] \cdots ]]$ 
is an extremal $n$-derivation.  
 
 Thus 
 $$
 \psi(x_1, x_2, \cdots, x_n)=\varphi(x_1, x_2, \cdots, x_n)- [x_1, [x_2, \cdots, , [x_n,  e\varphi(e, \cdots, e)f+(-1)^n f\varphi (e, \cdots, e)e] \cdots ]]
 $$ 
 is an $n$-Lie derivation which is also a central-valued mapping. With respect to the first component, 
 we have 
$$
\psi([x,y], ,x_2,\cdots,x_n)= [[\psi(x, ,x_2,\cdots,x_n),y] + [x, \psi(y, ,x_2,\cdots,x_n)] =0
$$
for all  for any $x, y, x_2, \dots, x_n \in G$. Hence, the $n$-linear central-valued mapping 
$\psi: \underbrace{G \times \cdots \times G}_n \longrightarrow \mathcal{G}$ vanishes on 
an arbitrary commutator $[x, y]$ in the first component. By the same computational approach and procedure, 
it can be shown that $\psi$ also vanishes on all commutators in the $i$-th\, ($2 \leq i \leq n$) component. 
We therefore conclude that the $n$-linear central-valued mapping $\psi$ vanishes on all commutators 
in every component. We eventually finish the proof of the first main theorem.
\end{proof}

As a consequence of Theorem \ref {xxsec4.1} and Lemma \ref{xxsec3.4} we have 
the following corollary.

\begin{tcolorbox}[breakable, blanker,left=3mm,right=3mm,
borderline west={2pt}{0pt}{blue}]

\begin{corollary}\label{xxsec4.2}
Let $\mathcal{G}=\left[
\begin{array}
[c]{cc}%
A & M\\
N & B\\
\end{array}
\right]$ be a generalized matrix algebra over a commutative ring $\mathcal{R}$
and $\varphi:\underbrace{\mathcal{G}\times \cdots\times \mathcal{G}}_n\longrightarrow\mathcal{G}$ be an
$n$-Lie derivation on $\mathcal{G}\, \, (n\geq 3)$ Suppose that
\begin{enumerate}
\item[{\rm (1)}] $\pi_{A}(\mathcal{Z}(\mathcal{G}))=\mathcal{Z}(A)$ and $\pi_{B}(\mathcal{Z}(\mathcal{G}))=\mathcal{Z}(B);$
\item[{\rm (2)}] either $A$ or $B$ does not contain nonzero central ideals;
\item[{\rm (3)}] if $\alpha a = 0, \alpha\in \mathcal{Z}(\mathcal{G}), 0 \neq a \in \mathcal{G},$ then $\alpha = 0;$
\item[{\rm (4)}] if $MN =0=NM$, then at least one of the algebras $A$ and $B$ is noncommutative;
\item [{\rm (5)}] every derivation on $\mathcal{G}$ is inner.
\end{enumerate}
Then $\varphi$ is of the form $\varphi=\kappa+\psi$, where $\kappa$ is an extremal $n$-derivation such that
$\kappa(x_1,x_2,\cdots,x_n) =[x_1,[x_2,[\cdots,[x_n, X_0]\cdots]]]$ for all $x_1, x_2, \cdots, x_n\in \mathcal{G}$, 
$\psi$ is an $n$-linear central-valued mapping vanishing on commutators in each component, 
$X_0=e\varphi(e,e,\cdots,e)f+(-1)^{n}f\varphi(e,e,\cdots,e)e$. 
\end{corollary}
\end{tcolorbox}

The following constitutes another significant theorem of this article:

\begin{tcolorbox}[breakable, blanker,left=3mm,right=3mm,
borderline west={2pt}{0pt}{red}]

\begin{theorem}\label{xxsec4.3}
Let $\mathcal{G}=\left[
\begin{array}
[c]{cc}%
A & M\\
N & B\\
\end{array}
\right]$ be a generalized matrix algebra over a commutative ring $\mathcal{R}$
and $\varphi:\underbrace{\mathcal{G}\times \cdots \times \mathcal{G}}_n\longrightarrow\mathcal{G}$ be an
$n$-Lie derivation on $\mathcal{G}\, \, (n\geq 3)$. Suppose that
\begin{enumerate}
\item[{\rm (1)}] $\pi_{A}(\mathcal{Z}(\mathcal{G}))=\mathcal{Z}(A)$ and $\pi_{B}(\mathcal{Z}(\mathcal{G}))=\mathcal{Z}(B);$
\item[{\rm (2)}] either $A$ or $B$ does not contain nonzero central ideals;
\item[{\rm (3)}] For each $n\in N$, the condition $Mn = 0 = n M$ implies $n=0$;
\item[{\rm (4)}] For each $m\in M$, the condition $Nm = 0 = m N$ implies $m=0$;
\item[{\rm (5)}] each special pair of bimodule homomorphisms has standard form.
\end{enumerate}
Then $\varphi$ has the form $\varphi=\kappa+\psi$, where $\kappa$ is an extremal $n$-derivation such that
$\kappa(x_1,x_2,\cdots,x_n) =[x_1,[x_2,[\cdots,[x_n, X_0]\cdots]]]$ for all $x_1, x_2, \cdots, x_n\in \mathcal{G}$, 
$\psi$ is an $n$-linear central-valued mapping vanishing on commutators in each component, 
$X_0=e\varphi(e,e,\cdots,e)f+(-1)^{n}f\varphi(e,e,\cdots,e)e$. 
\end{theorem}
\end{tcolorbox}

In view of \cite[Propositions 3.3 and 3.4]{DuWang3} and Corollary \ref{xxsec3.22},
we now get

\begin{tcolorbox}[breakable, blanker,left=3mm,right=3mm,
borderline west={2pt}{0pt}{blue}]

\begin{corollary}\label{xxsec4.4}
Let $\mathcal{G}=\left[
\begin{array}
[c]{cc}%
A & M\\
N & B\\
\end{array}
\right]$ be a generalized matrix algebra over a commutative ring $\mathcal{R}$
and $\varphi:\underbrace{\mathcal{G}\times \cdots\times \mathcal{G}}_n\longrightarrow\mathcal{G}$
be an $n$-Lie derivation on $\mathcal{G}\, \, (n\geq 3)$. Suppose that
\begin{enumerate}
\item[{\rm (1)}] $\pi_{A}(\mathcal{Z}(\mathcal{G}))=\mathcal{Z}(A)$ and $\pi_{B}(\mathcal{Z}(\mathcal{G}))=\mathcal{Z}(B);$
\item[{\rm (2)}] either $A$ or $B$ does not contain nonzero central ideals;
\item[{\rm (3)}] For each $n\in N$, the condition $Mn = 0 = nM$ implies that $n=0$;
\item[{\rm (4)}] For each $m\in M$, the condition $Nm = 0 = m N$ implies that $m=0$;
\item[{\rm (5)}] each derivation on $\mathcal{G}$ is inner.
\end{enumerate}
Then $\varphi$ is of the form $\varphi=\kappa+\psi$, where $\kappa$ is an extremal $n$-derivation such that
$\kappa(x_1,x_2,\cdots,x_n) =[x_1,[x_2,[\cdots,[x_n, X_0]\cdots]]]$ for all $x_1, x_2, \cdots, x_n\in \mathcal{G}$, 
$\psi$ is an $n$-linear central-valued mapping vanishingon commutators in each component, 
$X_0=e\varphi(e,e,\cdots,e)f+(-1)^{n}f\varphi(e,e,\cdots,e)e$. 
\end{corollary}
\end{tcolorbox}

As mentioned in Section \ref{xxsec2}, both full matrix algebras and triangular matrix algebras
are classic examples of generalized matrix algebras. Applying our main theorems to these backgrounds 
obtains the following corollaries.

\begin{tcolorbox}[breakable, blanker,left=3mm,right=3mm,
borderline west={2pt}{0pt}{blue}]

\begin{corollary}\label{xxsec4.5}
Let $M_{r\times r}(\mathcal{R})$ be the algebra of all $r\times r$ matrices
over a commutative ring $\mathcal{R}$, where $r>2$.
Then every $n$-Lie derivation on $M_{r\times r}(\mathcal{R})\, \, (n\geq 3)$ can be uniquely  decomposed into the sum of
an extremal $n$-derivation and an $n$-linear central-valued mapping vanishing on commutators in each component. 
\end{corollary}
\end{tcolorbox}

\begin{tcolorbox}[breakable, blanker,left=3mm,right=3mm,
borderline west={2pt}{0pt}{blue}]

\begin{corollary}\label{xxsec4.6}
Let $\mathcal{T}$ be a triangular algebra over a commutative ring $\mathcal{R}$ and
$\varphi:\underbrace{\mathcal{T}\times\cdots \times \mathcal{T}}_n\longrightarrow\mathcal{T}$ be an
$n$-Lie derivation on $\mathcal{T}\, \, (n\geq 3)$. Suppose that
\begin{enumerate}
\item [{\rm (1)}] $\pi_A(\mathcal{Z}(\mathcal{T}))=\mathcal{Z}(A)$ and $\pi_B(\mathcal{Z}(\mathcal{T}))=\mathcal{Z}(B)$;
\item [{\rm (2)}] either $A$ or $B$ does not contain nonzero central ideals;
\item [{\rm (3)}] each derivation on $\mathcal{T}$ is inner.
\end{enumerate}
Then $\varphi$ has the form $\varphi=\kappa+\psi$, where $\kappa$ is an extremal $n$-derivation such that
$\kappa(x_1,x_2,\cdots,x_n) =[x_1,[x_2,\cdots,[x_n, e\varphi(e,e,\cdots,e)f]]]$ and 
$\psi$ is an $n$-linear central-valued mapping vanishing on commutators in each component.
\end{corollary}
\end{tcolorbox}

\vspace{4mm}
\noindent{\bf Declaration of Competing Interest}
\vspace{3mm}

The authors declare that they have no known competing financial interests or personal relationships that could have appeared to influence the work reported in this paper.
\vspace{4mm}

\noindent{\bf Data availability} 
\vspace{3mm}

No data was used for the research described in the article.
\vspace{4mm}

\noindent{\bf Acknowledgements} 
\vspace{3mm}

The authors would like to express their thanks to the referees and the editors for their valuable comments and suggestions that improves the expositions clearly.
\vspace{4mm}

\noindent{\bf Funding} 
\vspace{3mm}

The work of the first author was supported by the Open Research Fund of Hubei Key Laboratory of Mathematical Sciences
(Central China Normal University), Wuhan ,430079, P.R. China,
Youth fund of Anhui Natural Science Foundation (Grant No.2008085QA01),
Key projects of University Natural Science Research Project of Anhui Province (Grant No. KJ2019A0107)

\bibliographystyle{amsplain}

\end{document}